\documentclass[12pt]{article}
\usepackage[english]{babel}
\usepackage{amsfonts,amsmath,amsxtra,amsthm,latexsym}
\usepackage{amssymb} 

\newtheorem{thm}{Theorem}[section]
 \newtheorem{cor}[thm]{Corollary}
 \newtheorem{lem}[thm]{Lemma}
 \newtheorem{prop}[thm]{Proposition}

 \theoremstyle{definition}
 \newtheorem{defn}{Definition}[section]
 \theoremstyle{remark}
 \newtheorem{rem}{Remark}[section]

 \newtheorem{ex}[thm]{Example}
 \numberwithin{equation}{section}

\DeclareMathOperator{\im}{Im}
\DeclareMathOperator{\re}{Re}
\DeclareMathOperator{\supp}{supp}
\DeclareMathOperator{\pr}{pr}
\DeclareMathOperator{\ind}{ind}

\DeclareMathOperator{\Bd}{bd}
\DeclareMathOperator{\Bdlocw}{\mathrm{loc-bd}_{\mathrm w*}}
\DeclareMathOperator{\cone}{cone}
\DeclareMathOperator{\Cone}{Cone}
\DeclareMathOperator{\conv}{conv}
\DeclareMathOperator{\Intr}{int}
\DeclareMathOperator{\argb}{\arg_0}
\DeclareMathOperator{\argph}{\arg_*}

\DeclareMathOperator{\Ray}{Ray}
\DeclareMathOperator{\Dr}{Dr}

\def\RR{\mathbb R}
\def\CC{\mathbb C}
\def\NN{\mathbb N}
\def\ZZ{\mathbb Z}
\def\TT{\mathbb T}
\def\BB{\mathbb B}
\def\DD{\mathbb D}
\def\PP{\mathbb{P}}
\def\QQ{\mathbb{Q}}
\def\Rnneg{\overline{\RR_+}}
\def\V{\mathbb{V}}

\def\vphi{\varphi}
\def\la{\lambda}
\def\de{\delta}
\def\ep{\varepsilon}
\def\ga{\gamma}
\def\om{\omega}
\def\Si{\Sigma}
\def\De{\Delta}
\def\pset{\mathcal{P}}

\def\pa{\partial}

\def\ii{\mathrm{i}}
\def\dd{\mathrm{d}}
\def\rr{\scriptstyle{r}}

\def\wt{\widetilde}
\def\dint{\displaystyle \int}

\def\modn{\hspace{-10pt}\mod}

\def\aone{a_*}
\def\I{\mathcal{I}}

\def\Mes{\mathbb{M}}
\def\A{\mathbb{A}}
\def\Am{\A_\Mes}
\def\Hyp{\mathrm{Hyp}}
\def\Cmp{\mathrm{Hcmp}}
\def\Sec{\mathrm{Sec}}
\def\Bone{\overline{\BB_1^+}}
\def\Bac{\overline{\BB_{1,\mathrm{ac}}^+}}
\def\DF{D^+}
\def\disc{\mathrm{disc}}
\def\Ext{\mathrm{Ext}}
\def\Tpl{\mathcal{T}}

\def\ChiCpl{\chi_{_{\scriptstyle \CC_+}}}

\begin{document}
\title{Pareto optimal structures producing resonances of minimal decay under $L^1$-type constraints}
\author{}
\date{}
\maketitle
{{\center {\large Illya M. Karabash } 
\\[4mm]

 Institute of Applied Mathematics and Mechanics of NAS of Ukraine,

 R. Luxemburg str. 74, Donetsk, 83114, Ukraine. \ \ 

E-mail addresses: i.m.karabash@gmail.com, karabashi@mail.ru


}}

\begin{abstract}
Optimization of resonances associated with
1-D wave equations in inhomogeneous media is studied under the constraint $\| B \|_1 \le m$ on
the nonnegative function $B \in L^1 (0,\ell)$ that represents the medium's structure.
From the Physics and Optimization points of view, it convenient to generalize
the problem replacing $B$ by a nonnegative measure $\dd M$ and imposing on $\dd M$ the condition
that its total mass is $\le m$.
The problem is to design for a given frequency $\alpha \in \RR$ a medium that generates a resonance
 $\om$ on the line $\alpha + \ii \RR$ with a minimal possible decay rate $| \im \om |$.
 Such resonances are said to be of minimal decay and form a Pareto frontier. We show that corresponding optimal measures consist
 of finite number of point masses, and that this result yields non-existence of optimizers for
  the problem over the set of absolutely continuous measures $B(x) \dd x$. Then we derive restrictions on optimal
  point masses and their positions.
  These restrictions are strong enough to calculate optimal $\dd M$ if the optimal resonance $\om$, the first point mass $m_1$,
  and one more geometric parameter are known. This reduces the original infinitely-dimensional problem to optimization
 over four real parameters. For small frequencies, we explicitly find the Pareto set
 and the corresponding optimal measures $\dd M$.
The technique of the paper is based on the two-parameter perturbation method and the notion of local boundary point.
The latter is introduced as a generalization of local extrema to vector optimization problems.
\end{abstract}

\quad\\
MSC-classes: 49R05, 58E17, 35B34, 34L15, 32A60, 58C06  

\tableofcontents

\section{Introduction}

Wave equations equipped with damping or radiation boundary conditions are used to model open resonators.
When separation of variables is possible, the Fourier decomposition leads to a non-self-adjoint spectral problem
that has a spectral parameter both in the equation and in the boundary conditions.
This parameter is usually called a \emph{quasi-(normal) eigenvalue} or a resonance.

\emph{Quasi-eigenvalues}
considered in this paper are the values of the spectral parameter $\om \in \CC \setminus \{ 0\}$
such that the problem
\begin{eqnarray}
- \pa_x^2 y(x)  = \om^2 B(x) y (x)  , \ \ \ \
\pa_x y (0)  =  0 , \ \ \ \
y (\ell) = - \ii \pa_x y (\ell)/\om  , \label{e epB}
\end{eqnarray}
has a nontrivial solution $y$ on the interval $\I = [0,\ell]$.
A natural generalization of this problem
\begin{eqnarray} \label{e dMdx eq}
- \frac{\dd^2}{\dd M \dd x} y(x) & = & \om^2 y(x) , \\
 \pa_x^- y (0) & = & 0, \label{e BC0} \\
\ \ y(\ell) + \frac{i}{\om} \, \pa_x^+ y(\ell)& = & 0 , \label{e BCl}
\end{eqnarray}
involves the Krein-Feller differential expression $\frac{\dd^2}{\dd M \dd x} $, see \cite{KK68_II,DM76} and Section \ref{ss q-e}.
It corresponds to the change of the nonnegative function $B $ in (\ref{e epB}) to a nonnegative measure $\dd M$. When $\dd M = B \dd x$ with the density $B \in L^1 (0,\ell)$, problem (\ref{e dMdx eq})-(\ref{e BCl}) turns into (\ref{e epB}).

These models are relevant to many physical
systems with 1-D, multilayered, or radially symmetric structures. These include vibrations of a damped string with the mass distribution $\dd M$ (or the density
$B$) \cite{A75,KN79,KN89,CZ95,P97,PvdM01}, the Regge problem \cite{GP97,PvdM01}, and standing EM waves in an open optical cavity with
a symmetric 1-D structure \cite{LLTY94,MvdBY01,HBKW08,Sch11,Ka13_Optic}. In the latter case,
the cavity is often called a 1-D photonic crystal and is described by the relative permittivity function $B$.
Very kindred problems arises in connection with resonances or scattering poles
of Schr\"odinger equation \cite{GK71,HS85,F97},
standing acoustic waves  \cite{Sh96}, and gravitational radiation from a star \cite{LLTY94,MvdBY01}.

In Physics, problem (\ref{e dMdx eq})-(\ref{e BCl}) usually involves a measure $\dd M$
of the form $B(x) \dd x + \sum_{j=1}^n m_j \delta (x-a_j) \dd x$.
Here $B(x) \dd x$ is an absolutely continuous part of $\dd M$, and
$m_j \de (x-a_j) \dd x$ is a point mass $m_j$ positioned at $x=a_j$.
In Optical Engineering, the term $m_j \delta (x-a_j) \dd x$ corresponds
to a thin layer of high relative permittivity forming
a partially transmitting dielectric mirror \cite{LSL73,LLTY94,MvdBY01}.

The above models involve either damping, or leakage of energy into surrounding medium.
For  standing waves associated with (\ref{e dMdx eq})-(\ref{e BCl}),
this leads to exponential decay in time and to the fact that quasi-eigenvalues $\om$
lie in the open lower half-plane $\CC_- $ \cite{A75,KN79,KN89}.
\emph{The decay rate} $\Dr (\om,\dd M)$ of standing waves associated with $\om$ and $\dd M$
equals the minus imaginary part of $\om$  and is always positive. The real part $\re \om$
is the \emph{frequency} of oscillations. The set of quasi-eigenvalues associated
with the problem (\ref{e dMdx eq})-(\ref{e BCl}) is denoted by $\Si (\dd M)$.
The associated eigenfunctions are called (quasi-)modes (metastable states, in the context of Quantum Mechanics).

The recent engineering progress in design of resonators with small and high decay rate
\cite{AASN03,V03,NKT08} attracted considerable current interest to numerical aspects of emerging optimization problems
\cite{HBKW08,KS08,BPSchZsch11,OW11,Sch11,PMBLKR12}.

This paper is devoted to optimization of an individual quasi-eigenvalue under \emph{total mass constraints} on the coefficient $B$ or $\dd M$ and, in particular, to minimization of
 the decay rate $|\im \om|$. Optimization is considered over the set of measures (admissible family)
\begin{equation}
 \Am :=   \{ \dd M  \ : \ \dd M \text{ is a nonnegative Borel measure such that } \int_{0-}^{\ell+} \dd M \le m \} , \ \ m>0, \label{e Ad mes}
\end{equation}
 and (as a by-product) over
\begin{equation}
\A_1 :=   \{ B(x) \dd x \in \Am \ : \ B \in L^1 (0,\ell) \} \text{ and }
\A_{\disc} := \left\{ \sum_{j=1}^n m_j \delta (x-a_j) \dd x \in \Am : \ n \in \NN \right\} .
\label{e Ad1 Adisc}
\end{equation}
Note that numerical methods were used in \cite{S87} and, recently, in
\cite[Problem $\mathrm{Opt_{area}}$]{HBKW08} to consider very kindred optimization problems (with slightly different Physics backgrounds).
From Spectral Theory point of view, the optimization over the above sets can be considered as an attempt to obtain sharp estimates
on quasi-eigenvalues in terms of $\| B \|_1$ or the norm $\| \dd M \|$ of $\dd M$
(such estimates are a classical topic in Mathematical Physics \cite{H92}).

Our  main attention is directed to rigorous derivation of structural theorems for optimizers.
 The motivation is that an accurate numerical discretization requires some understanding of optimizers' structure.
The structural optimization for quasi-eigenvalues and, more generally,
 infinite-dimensional non-self-adjoint spectral problems is not adequately developed.

Short reviews of applied and numerical studies relevant to quasi-eigenvalue optimization can be found in \cite{OW11} and the monograph
\cite{Sch11}.
Most of the analytical literature on non-self-adjoint spectral optimization is concerned with optimization of spectral abscissa for damped systems.
This problem involves simultaneous optimization of all eigenvalues and, in infinite-dimensional case, becomes very difficult and quite different \cite{CZ94,CZ95,CO96}.
Though the finite-dimensional theory is more developed (e.g. \cite{BO94,FL99}),
the structural study of optimizers is not a simple question \cite{FL99}.

The essential difficulty for the standard critical point argument arises when eigenvalues are not differentiable with respect to (w.r.t.) the coefficients. Such effects usually appears
when eigenvalue is non-simple (degenerate). In this case, perturbations are studied with the use of
Puiseux series and Vishik-Lyusternic-Lidskii theory.
Relevant references and connections with the optimization of spectral abscissa can be found in \cite{CO96}
(where an explicit example of a non-differentiable splitting for a multidimensional dissipative wave equation is given), in \cite{BO94,FL99} (matrix case),
and in the authors papers \cite{Ka13,Ka13_Optic} concerning optimization of an individual quasi-eigenvalue
under side constraints on $B$.
Examples related to degeneration and merging of 1-D quasi-eigenvalues are known in Physics and Spectral Theory,
see \cite{MvdBY01,Ka13_KN} and the discussion in Section \ref{ss Mult and ex}.
It follows from results of Section \ref{ss PEig} that non-semi-simple
quasi-eigenvalues of (\ref{e dMdx eq})-(\ref{e BCl}) are \emph{always non-differentiable} w.r.t. $\dd M$,
see (\ref{e k Pser}) and Example \ref{ex 1 pt mas}.
According to \cite{CO96,MvdBY01}, these degeneracy and non-differentiability issues  cannot be discarded as mathematical exotic and can have experimental consequences (e.g. diverging laser quantum noise \cite{vdLvDMvELvEW97}).

Another difficulty arises since the standard proof of existence of optimizers relies on compactness
 arguments \cite{K51} and is applicable only to weakly* compact admissible families.
The latter is the case for $\Am$, but not for $\A_1$ since the space $L^1 (0,1)$ has no predual.
In fact, one of the results of this paper states that there are no minimizers for the decay rate over $\A_1$
 (see Theorem \ref{t A1} for the precise statement). On the other hand, the adaptation of compactness
 arguments to dissipative problems uses an additional restriction $\alpha \in [\alpha_1,\alpha_2]$
  on the real part of the  spectral parameter \cite{HS85,S87,Ka13}.
  As a result, the optimizer is not necessarily a critical point of the decay rate functional $\Dr (\om,\dd M)$.
   This makes it difficult to use the standard one-parameter perturbation theory in the study of optimizers' structures, see the discussions in \cite{Ka13,Ka13_Optic}.

In the author's paper \cite{Ka13}, the two-parameter perturbation approach was developed to overcome
the above obstacles in the study of optimization over the family
\[
\A_\infty := \{ B(x) \dd x \in \Am \ : \ B \in L^\infty (0,\ell) \text{ and } c_1 \le B (x) \le c_2 \text{ a.e.} \},  \ \ 0 \le c_1 < c_2 .
\]
 It was shown that structures of minimal decay are extreme points of
$\A_\infty$ and, moreover, are piecewise constant functions taking only extreme possible values.
Each optimal $B$ is related to one of corresponding optimal modes $y$
via $B (x) = c_1 + (c_2 - c_1) \ChiCpl \left(y^2 (x) \right)$,
where $\ChiCpl (z) := 1$ for $z \in \CC_+ $, and $\ChiCpl  (z) := 0$ for $z \in  \CC \setminus \CC_+$.
This leads to the conclusion that optimal $\om$ are eigenvalues of nonlinear equation
$ y'' + \om^2 y  \left[ c_1 + (c_2 - c_1) \ChiCpl (y^2 ) \right] =  0 $
equipped with the boundary conditions (\ref{e BC0})-(\ref{e BCl}).

The goals of the present paper require further refinement of the two-parameter perturbation method
and the essential use of convex analysis and geometric arguments.
Following \cite{Ka13}, Section \ref{s OptProbl} gives the rigorous statement of the optimization problem in
the framework of vector (Pareto) optimization.
Roughly speaking, a quasi-eigenvalue $\om$ itself is considered as an $\RR^2$-valued cost function depending on  $\dd M$.
(Actually, the formalization of $\om (\dd M)$ leads to the set-valued map $\dd M \mapsto \Si (\dd M)$.
It is multivalued even locally due to the coalescence and splitting issues.)
\emph{Quasi-eigenvalues of minimal decay} are defined as the points of the Pareto set.
Note that this provides sharp bounds on quasi-eigenvalues in terms of
$\| \dd M \|$ when the Pareto set is found (see Theorem \ref{t small fr}).

Section \ref{s properties} contains basic properties and definitions concerning quasi-eigenvalues,
as well as preparative results including: analyticity and derivatives of the characteristic determinant $F$, discussion of multiplicity and one-parameter perturbation formulae for quasi-eigenvalues, and an example of a degenerate and non-differentiable quasi-eigenvalue.

In Section \ref{s 2par pert}, we introduce the notion of \emph{local boundary point} for images of set-valued maps.
Local boundary points are introduced as an extension of the notion of local extrema to vector optimization problems
with two criteria (these are the frequency $\alpha= \re \om$ and the decay rate $\beta = - \im \om$).
 Alternatively, the set of local boundary point can be considered as the local version of the Pareto set.
Theorem \ref{t ab loc bd} reduces the study of local boundary points to convex analysis of the $x$-trajectory of the mode. The proof is based on two-parameter perturbations of Lemma \ref{l 2par per}. This lemma deals with a zero surface of an analytic function of three complex variables in a neighbourhood of a degenerate root.
The proof of Lemma \ref{l 2par per} refines that of \cite[Lemma 3.6]{Ka13} with the use of homotopy arguments.

The rest of the paper is essentially concerned with the study of local boundary point of the
\emph{set of admissible quasi-eigenvalues} $\Si [\Am] := \bigcup_{\dd M \in \Am} \Si (\dd M) $.
The analysis of local boundary points is convenient since they are independent of the choice of
 the function of $\om$ that has to be optimized. In particular, for each frequency $\alpha$,
 the quasi-eigenvalue of minimal decay rate for $\alpha$ is a local boundary point of $\Si [\Am]$.

It is easy to see that optimizers over $\Am$ are not necessarily extreme points of $\Am$.
Indeed, all quasi-eigenvalues produced by the extreme points $m \de (x -x_0) \dd x$ belong to the circle  $\{z \in \CC \ : \  |z + \ii/m | = 1/m\}$ and so can not be optimal for high frequencies.
Theorem \ref{t AM opt} state that optimal measures belongs to finite-dimensional faces of $\Am$.
In other words, they consist of finite number of point masses. An example of an optimizer that is not an extreme point of $\Am$ is given in Section \ref{s c rem}. So the optimization over
$\Am$ is equivalent to the same problem over $\A_{\disc}$.
The absence of optimizers over $\A_1$ (Theorem \ref{t A1}) is obtained as a by-product of this study.

In Sections \ref{ss strong loc-b}
and \ref{s weak loc-b}, we derive restrictions on the point masses and their positions. These restrictions are strong enough to calculate optimal $\dd M$
if the optimal quasi-eigenvalue $\om$, the first point mass $m_1$, and one more geometric parameter are known. In Section \ref{s small fr}, this allows us to explicitly find the Pareto frontier and corresponding optimal measures $\dd M$ for small frequencies.

The optimization problems of this paper can be considered as problems with two constraints,
one for the total mass and the other for the length of the interval $\I=[0,\ell]$.
In Section \ref{s c rem}, we show that at least one of these constraints is achieved by every optimizer.\\[1mm]

\noindent \textbf{Notation}. We use the convention that a sum equals zero if the lower index exceeds
the upper.

The following sets of real and complex numbers are used:
open half-lines $\RR_\pm = \{ x \in \RR: \pm x >0 \}$, open half-planes $\CC_\pm = \{ z \in \CC : \pm \im z >0 \}$,
open discs $\DD_\epsilon (\zeta) := \{z \in \CC : | z - \zeta | < \epsilon \}$ with the center at $\zeta$
and radius $\epsilon$, the unit circle $\TT :=  \{z \in \CC : | z | = 1\}$, the infinite sector (without the origin $z=0$)
\begin{equation} \label{e Sec def}
\Sec [\xi_1,\xi_2] := \{ z \in \CC\setminus{0} \, : \, \arg z = \xi \ (\modn 2\pi)
 \text{ for certain } \xi \in [\xi_1, \xi_2] \} ,  \ \ \ \xi_1 \le \xi_2 , \ \ \ \xi_{1,2} \in \RR ,
\end{equation}
closed line-segments $[z_1,z_2]$ between endpoints $z_{1,2} \in \CC$, and
line-segments with excluded endpoints $(z_1,z_2) := [z_1,z_2] \setminus \{z_1,z_2\}$.
A line segment is called \emph{degenerate} if its endpoints coincide.
$\QQ_I$, $\QQ_{II}$, $\QQ_{III}$, and $\QQ_{IV}$ are the open quadrants in $\CC$.

Let $S$ be a subset of a linear space $U$ over $\CC$ (including the case $U=\CC$).
For $u_{0,1} \in U$ and $z \in \CC$,
\[
z S  + u_0 := \{ zu + u_0 \, : \, u \in S \}  \ \ \text{ and } \ \
[u_0,u_1] := \{ (1-\la) u_0 + \la u_1 \ : \ \la \in [0,1] \} .
\]

The closure of a set $S$ (in the norm or the Euclidean topology) is denoted by $\overline{S}$,
 in particular, $\overline {\RR_-} = (-\infty,0] $, $\Rnneg = [0,+\infty)$,
 and $\Rnneg^N = [0,+\infty)^N$ is the set of $N$-tuples of nonnegative numbers.
$\Intr S$ and $\Bd S$ are the sets of interior and boundary points of $S$, respectively.

For basic definitions of convex analysis we refer to \cite{RW98}.
The convex hull of $S$ is denoted by $\conv S$.
The convex cone generated by the set $S$ (the
set of all nonnegative linear combinations of elements of $S$) is denoted by $\cone S$.

For open balls in a normed space $U$, we use
$
\BB_\epsilon (u_0) =
\BB_\epsilon (u_0; U) := \{u \in U \, : \, \| u - u_0 \|_U < \epsilon
\}.
$

By $\Mes_{\CC}$ and $\| \cdot \|$ we denote the Banach space of complex
Borel measures on $\I=[0,\ell]$ and the corresponding norm.
$\Mes_+$ is the cone of nonnegative measures in $\Mes_\CC$.
The set of nonnegative measures
in the closed unit ball is denoted by
$ \Bone := \{ \dd M \in \Mes_+ \ : \ \| \dd M \| \le 1 \} $ .

For $\dd M \in \Mes_\CC$ and a
Borel set $S$,
define the projection $\pr_{S} \dd V $
of $\dd V$ to $S$ as the measure that coincides with $\dd V$ on $S$
and coincides with the zero-measure on $\RR \setminus S$.
The topological
support $\supp \dd M $ of $\dd M $ is the smallest closed set
$S $ such that $\| \pr_{\RR \setminus S } \dd M \| = 0$.
By
$ \int_{a\pm}^{b\pm} f(x) \dd M (x)$
the corresponding Lebesgue--Stieltjes integrals over the intervals
$(a,b]$, $(a,b)$, $[a,b]$, and $[a,b)$ are denoted.
The notation $0 \dd x$ ($1 \dd x$) means the zero measure on $\I$ (resp., the Lebesgue measure on $\I$).

$L^p (0,\ell) $ are the Lebesgue spaces of complex valued functions and
\[
W^{k,p} [0,\ell] := \{ y \in L^p (0,\ell) \ : \ \pa_x^j y \in L^p (0,\ell), \ \ 1 \leq j \leq k \}
\]
are Sobolev spaces with standard norms $\| \cdot \|_{L^p}$ and $\| \cdot \|_{W^{k,p}}$. The space of continuous
complex valued functions with the uniform norm is denoted by $C
[0,\ell]$.

For a function $f$ defined on a set $S$, $f[S]$ is the image of $S$.
By $\pa_x f$, $\pa_z f $, etc. denote (ordinary or partial)
derivatives with respect to (w.r.t.) $x$, $z$, etc.

We write $z_1^{[n]} \asymp z_2^{[n]}$ as $n \to \infty$ if the sequences
$z_1^{[n]} / z_2^{[n]} $ and $z_2^{[n]} / z_1^{[n]}$ are bounded for $n$ large enough.

Sometimes, the complex plane $\CC$ is considered as a real linear space $\RR^2$ with the scalar product
\begin{equation} \label{e <>C}
\langle z_1 , z_2 \rangle_\CC := \re z_1 \re z_2 + \im z_1 \im z_2  , \ \ \ z_{1,2} \in \CC .
\end{equation}

\section{The statement of optimization problem}
\label{s OptProbl}

\subsection{Quasi-eigenvalues of Krein strings}
\label{ss q-e}

Let $\Mes_+$ be the set of bounded nonnegative Borel measures on the interval $\I = [0,\ell]$ of a finite positive length $\ell$.
In the settings of this paper, \emph{a (Krein) string} is the
interval $[0,\ell]$ carrying a dispersed mass, which is
represented by a measure $\dd M \in \Mes_+$.

With the string one can associate the quasi-eigenvalue problem (\ref{e dMdx eq})-(\ref{e BCl}).
To define the left- and right-hand derivatives $ \pa_x^- y (0)$ and
$\pa_x^+ y (\ell)$ in (\ref{e BC0}) and (\ref{e BCl}), it is convenient to assume that
\begin{eqnarray}
& \dd M \text{ is continued to } (-\infty,0) \text{ and } (\ell,+\infty)
\text{ by the zero measure} \label{a contM} \\
& \text{ and } y \text{ satisfies (\ref{e dMdx eq}) on $\RR$ \qquad (see \cite{KK68_II,A75,DM76} for details).} \label{a cont y}
\end{eqnarray}

\begin{rem}
Actually, each locally bounded Borel continuation of $\dd M$ to a vicinity of $[0,\ell]$ provides the
same values of $ \pa_x^- y (0)$ and $\pa_x^+ y (\ell)$. The continuation of $\dd M$ to $(\ell, \ell_1]$  by the Lebesgue
measure was used in \cite{KN89} since with such a special continuation the change of $\ell$ to $\ell_1$ in (\ref{e BCl})
saves the positions of quasi-eigenvalues.
\end{rem}

Assuming that $\dd M$ is extended by (\ref{a contM}), we define
\begin{equation} \label{e a2}
 \aone (\dd M) := \sup \{ x \in (-\infty,\ell] \ : \ \text{the restriction of }
 \dd M(s) \text{ on } (-\infty,x) \text{ is the zero measure} \}.
\end{equation}

The eigenvalue problem (\ref{e dMdx eq})-(\ref{e BCl})
 corresponds to free
transverse harmonic oscillations of a string with the left end $x=0$
sliding without friction and the damped right end $x=\ell$. The damping force is proportional to the velocity
of motion. 
A simple way to understand (\ref{e dMdx eq})-(\ref{e BCl}) is to rewrite the problem in the integral equation form
 \begin{gather} \label{e int eq}
 y(x) = y (0) - \om^2 \int_{0}^x \dd t \, \int_{0-}^{t-} y (s) \ \dd M (s) , \\
 y(\ell) -  \ii \om \int_{0-}^{\ell+}  y (s) \ \dd M (s) = 0. \label{e int bc}
 \end{gather}

In this paper, the boundary conditions and the length $\ell$ of the interval $\I$ are fixed. So the string is
completely determined by the measure $\dd M$. Therefore we will speak about \emph{the string} $\dd M$.

Eigen-parameters $\om \in \CC \setminus\{0\}$ such that (\ref{e dMdx eq}) has a
nontrivial continuous solution $y $ will be called \emph{quasi-eigenvalues} of $\dd M$. (Recall that a solution $y$
is said to be \emph{nontrivial} if $y \not \equiv 0$ on $[0,\ell]$).
The real part $\alpha=\re \om $ of the quasi-eigenvalue $\om$ characterizes the frequency
of oscillations corresponding to the mode $y$. For simplicity, $\alpha$ will be called the \emph{frequency} of $\om$
(actually, the angular frequency of the oscillations equals $|\alpha|$).
The minus imaginary part $\beta = - \im \om$ is always positive and
characterizes the \emph{rate of decay} of the oscillations.

The set of quasi-eigenvalues $\om$ of a string $\dd M$ is denoted by
$\Si (\dd M) $. It is known that $\Si (\dd M) \subset \CC_-$, that
quasi-eigenvalues are isolated, and that $\infty$ is their only
possible accumulation point \cite{KN79,KN89} (see also \cite{CZ95}).
The case $\om = 0$ is mathematically and physically special
and is usually excluded from the considerations, see the explanations in \cite{CZ95} and Section  \ref{ss ComplString}.

\begin{ex} \label{ex empty}
If $\dd M = 1 \dd x$ (the Lebesgue measure on $[0,\ell]$) or $\dd M = 0 \dd x$ (the zero measure), then
$\Si (\dd M) $ is empty.
\end{ex}

\subsection{Minimal decay rate}

The aim of the paper is to study optimization of quasi-eigenvalues of strings with constrained total mass and fixed length $\ell$ of the interval $\I$.
We consider optimization over the \emph{admissible families} $\A_1$ and $\Am$ defined by
(\ref{e Ad mes})-(\ref{e Ad1 Adisc}) with a constant $m>0$.


In the following, the admissible family $\A$ is either $\A_1$, or $\A_\Mes$, if not said otherwise.
The strings in $\A$ are called \emph{admissible}.
We say that a complex number $\om$ is \emph{an admissible quasi-eigenvalue} if it belongs to the set
$ \
\Si [\A] := \bigcup_{\dd M \in \A} \Si (\dd M) , \
$
and say that $\alpha \in \RR$ is \emph{an admissible frequency} if $\alpha = \re \om$ for some
admissible quasi-eigenvalue $\om$.

\begin{defn}[\cite{Ka13}]
Let $\alpha$ be an admissible frequency for an admissible family $\A$.

\item[(i)] \emph{The minimal decay rate $\beta_{\min} (\alpha)= \beta_{\min} (\alpha; \A )$ for the frequency} $\alpha$
is defined by
\[
\beta_{\min} (\alpha) := \ \inf \{ \beta \in \RR \ : \ \alpha - \ii \beta  \in \Si [\A] \} .
\]

\item[(ii)] If $\om = \alpha - \ii \beta_{\min} (\alpha)$ is a quasi-eigenvalue of certain admissible string
$\dd M \in \A$ (i.e., the infimum is achieved), we say that $\om $ and $\dd M$ are of \emph{minimal decay for the frequency} $\alpha $
(over the admissible family $\A$).
\end{defn}

\section{Properties of quasi-eigenvalues, modes, and related maps}
\label{s properties}

When a number $z$ or a string $\dd M $ is fixed we will write simply $\vphi (x)$ or $\vphi (x,z) $ instead of
$\vphi (x,z; \dd M_0)$ and will use the same shortening for 
the number $\aone$ defined by (\ref{e a2}).

\subsection{Complex Krein strings and the characteristic determinant $F$}
\label{ss ComplString}

It is convenient for us to generalize problem (\ref{e dMdx
eq})--(\ref{e BCl}) and the definition of quasi-eigenvalues to measures $\dd M$ from the class $\Mes_{\CC}$ \emph{of complex
Borel measures} on $[0,\ell]$. Recall that $\Mes_{\CC}$ is a complex
Banach space with the norm $\| \dd M \| := \int_{0-}^{\ell+} |\dd M | $,
where $|\dd M| $ denotes \emph{the total variation measure} of $\dd
M$.



Assuming that $\dd M$ is extended according to (\ref{a contM}),
we denote by $\vphi (x ) = \vphi (x, z; \dd M )$ and $\psi(x) = \psi (x, z; \dd M)$  the solutions of
$\frac{\dd^2}{\dd M \dd x}  y(x) = - z^2 y(x) $ on $\RR$ satisfying
\begin{equation*} 
\vphi (0, z; \dd M ) = \pa_x^- \psi (0, z; \dd M) = 1 , \ \ \pa_x^-
\vphi (0, z; \dd M ) = \psi (0, z; \dd M) = 0 .
\end{equation*}

The function $\vphi$ is absolutely continuous in each bounded interval
and have the left and right derivatives $\pa_x^\pm \vphi $ at every $x \in \RR$.
Moreover,
$\lim_{x \to x_0 +} \pa_x^\pm \vphi (x) = \pa_x^+ \vphi (x_0),$ \
$\lim_{x \to x_0 -} \pa_x^\pm \vphi (x) = \pa_x^- \vphi (x_0).$
The same holds for the function $\psi$, see \cite{KK68_II} for details.

Obviously,
$\vphi (x,z ; \dd M)$ is a unique solution to the
integral equation
\begin{eqnarray}
 y (x) & = & 1 - z ^2 \ \int_{0-}^x (x-s) \ y (s) \ \dd M (s) ,
\quad 0 \leq x \leq \ell. \label{e int ep phi}
\end{eqnarray}

The following lemma is a rigorous form for the integral reformulation of the quasi-eigenvalue problem.

\begin{lem}[see e.g. \cite{A75}]\label{l int}


A number $\om \in \CC$ belongs to $\Si (\dd M)$ if and only if there exists nontrivial $y (x) \in C [0,\ell]$
satisfying (\ref{e int ep phi}) with $z = \om$ and (\ref{e int bc}).
\end{lem}

Note that in the integral settings there is no need to exclude separately the case $\om = 0$.
When $\om =0$, it is easy to see that (\ref{e int eq}), (\ref{e int bc}) has no nontrivial solutions,
see the proof of \cite[Lemma 2.7]{Ka13}.

Consider the functional
\begin{equation} \label{e F}
F (z; \dd M) := \vphi (\ell,z; \dd M)  - \ii z
\int_{0-}^{\ell+}  \vphi (s, z; \dd M) \ \dd M (s) ,
\end{equation}
produced by (\ref{e int eq}).
When $z \neq 0$, one has
\[
F (z; \dd M) = \vphi (\ell,z; \dd M) + \frac{i}{z} \, \pa_x^+ \vphi (\ell,z; \dd M) .
\]
From Lemma \ref{l int} or directly from the statement of the quasi-eigenvalue problem we see that
$\om \in \Si (\dd M) \text{ exactly when } F (\om ; \dd M) = 0.$

We say that a map $G : U_1 \to U_2$ between normed spaces $U_{1,2}$
is \emph{bounded-to-bounded} if the set $G[S]$ in $U_2$ is bounded for
every bounded set $S$ in $U_1$. Basic facts about analytic maps on Banach
spaces can be found in \cite{PT87}.

\begin{lem} \label{l an}
$(i)$ The functional $F (z; \dd M) $
is analytic on $\CC \times \Mes_\CC$.
In particular, $\Si (\dd M)$ is the set of zeroes of the entire function $F (\cdot) = F (\cdot ;
\dd M)$.

$(ii)$ The map $(z, \dd M) \mapsto \vphi (\cdot ,z; \dd M)$ is
bounded-to-bounded and analytic from $\CC \times \Mes_\CC$ to \linebreak
$W_\CC^{1,\infty} [0,\ell]$. Its Maclaurin series is given by
\begin{eqnarray}
\vphi (x , z ; \dd M) = 1 - \vphi_1 (x; \dd M) z^2 + \vphi_2 (x; \dd M) z^4 - \vphi_3 (x; \dd M)  z^6 + \dots , \label{e pow ser}\\
\vphi_0 (x; \dd M) \equiv 1 , \ \  \vphi_{j} (x; \dd M) = \int_{0-}^x
(x - s) \, \vphi_{j-1} (s; \dd M) \, \dd M (s) , \ \ j \in \NN .
\label{e phij}
\end{eqnarray}

\end{lem}

\begin{proof}
\textbf{(i)} follows immediately from (ii). Let us prove
\textbf{(ii)}. Consider (\ref{e pow ser}) as a $C [0,\ell]$-valued
series. One can see that $|\vphi_{j} (x;\dd M)| \le \vphi_j (x; |\dd
M|)  \le \vphi_j (\ell; |\dd M|)$ for all $x \in [0,\ell]$. By
\cite[Section 2]{KK68_II} (see also \cite[Exercise 5.4.2]{DM76}), one has
\begin{equation} \label{e phij<}
\vphi_{j} (\ell; |\dd M|) \leq \frac{(2\ell \| \dd M \|_\Mes)^j}{(2j) !} , \ \ \ j \in \NN  .
\end{equation}
Hence, the series (\ref{e pow ser}) converge uniformly on every bounded set of
$\CC \times \Mes_\CC$. So  (\ref{e pow ser}) defines an analytic map
from $\CC \times \Mes_\CC$ to $C [0,\ell]$. This map is also
bounded-to-bounded due to (\ref{e phij<}). Since (\ref{e pow ser})
satisfies (\ref{e int ep phi}), series (\ref{e pow ser}) is the Maclaurin
series of $ \vphi$.
By  (\ref{e int ep phi}),
\begin{eqnarray} \label{e pa phi 1}
\pa_x^\pm \vphi (x) & = & - z^2 \ \dint_{0-}^{x\pm} \vphi (x) \dd M (s) ,
\ \ x \in (0,\ell], \\
\text{ and } \ \ \pa_x \vphi (x) & = & \pa_x^+ \vphi (x) = \pa_x^- \vphi (x) \text{ a.e. on }(0,\ell). \label{e pa phi 2}
\end{eqnarray}
This implies $\pa_x \vphi \in L_\CC^\infty (0,\ell) $. Plugging  (\ref{e
pow ser}) into (\ref{e pa phi 1})-(\ref{e pa phi 2}), one gets an $L_\CC^\infty$-valued
Maclaurin series for $\pa_x \vphi $. Due to (\ref{e phij<}), this
series converges uniformly on every bounded set of $\CC \times \Mes_\CC$
to a bounded-to-bounded map from $\CC \times \Mes_\CC$ to $L_\CC^\infty
(0,\ell)$. This proves (ii).
\end{proof}

For Sturm-Liouville equations, formal expansions of such type  were already known to Hermann Weyl,
while the analyticity of fundamental solution w.r.t. the coefficients of the equation was emphasized
and intensively used in \cite{PT87}.

\subsection{Multiplicities and examples of degenerate quasi-eigenvalues}
\label{ss Mult and ex}

It is obvious that all modes $y (\cdot)$ corresponding to $\om \in \Si (\dd M)$
are equal to $\vphi (\cdot, \om; \dd M )$ up to a multiplication by a
constant.  So \emph{the geometric multiplicity} of any
quasi-eigenvalue equals 1. In the following, \emph{the multiplicity}
of a quasi-eigenvalue means its \emph{algebraic multiplicity}.

 \begin{defn}\label{d mult}
\emph{The multiplicity} of a quasi-eigenvalue is its multiplicity as
a zero of the entire function $F (\cdot; \dd M) $. A quasi-eigenvalue is
called \emph{simple} if its multiplicity is $1$.
 \end{defn}

This definition was used in \cite{A75,KN79,KN89} for nonnegative $\dd M$ and
is a natural extension of the classical M.V. Keldysh
definition of multiplicity for eigenvalue problems with an
eigen-parameter in boundary conditions \cite{N69}. Indeed, when $\dd M = B \dd x$,
the function $F(z)= F (z; \dd M)$ is equal to the characteristic determinant \cite{N69}
of (\ref{e dMdx eq})-(\ref{e BCl}) up to a multiplication on a nonzero constant.

Since $F (0; \cdot ) \equiv 1$, each quasi-eigenvalue has a finite
multiplicity. There exist strings $\dd M\in \Mes_+$ with \emph{multiple quasi-eigenvalues} (i.e., quasi-eigenvalues of multiplicity $\ge 2$).
 A simple example that fits the settings of the present paper was
given recently in \cite[Remark 2.1]{Ka13_KN} (see formula (\ref{e x0 ne l}) below).
For slightly different classes of strings, the existence of multiple quasi-eigenvalues
can be obtained from the results on the direct spectral problem for quasi-eigenvalues
\cite[Theorem 3.1]{KN89}, \cite{GP97}, and
\cite[Theorem 4.1]{PvdM01} and was explicitly noticed in \cite{MvdBY01} (the last paper provides an example with the boundary condition $y(0)=0$ instead of (\ref{e BC0})).

\begin{ex}[\cite{Ka13_KN}] \label{ex 1 pt mas}
Let $\dd M $ be the string consisting of a mass $m_0 > 0$
placed at a point $x_0 \in [0,\ell] $,
i.e., writing with  Dirac's $\delta$-function
$ \dd M = m_0 \delta (x-x_0) \dd x $.
Then the set of quasi-eigenvalues of $\dd M$ has the following description:
\begin{eqnarray}
& \Si (\dd M )  = \{ - \ii m_0^{-1} \} \ \ & \text{ if } \ \ x_0 = \ell \label{e x0=l}, \\
& \Si (\dd M )  = \left\{ \ - \ii \frac {1}{2 (\ell - x_0) }
\pm \sqrt{ \frac{1}{m_0 (\ell - x_0)} -
\frac {1}{4 (\ell - x_0)^2} }  \ \right\} \ \ & \text{ if } \ \ x_0 < \ell. \label{e x0 ne l}
\end{eqnarray}
These formulae take into account the multiplicities. That is, when $4 (\ell - x_0) =m_0$,
(\ref{e x0 ne l}) means that $ \frac {-\ii}{2(\ell - x_0)}$ is a quasi-eigenvalue of multiplicity 2.
In the other cases, each quasi-eigenvalue is simple.
\end{ex}

\subsection{Trajectories of $\vphi$-solutions}

\begin{lem} \label{l phi R}
$(i)$ \ $ \vphi (x,z,\dd M) =1 $ for all $x \in [0,\aone ]$. \\
$(ii)$ If $\dd M$ is a nonnegative measure, $z^2 \not \in \RR$, and
$x \in \left( \aone , \ell \right]$, then
\begin{equation} \label{e im vhi vhi'}
\im z^2 \ \im [\overline{\vphi (x,z)} \pa_x^\pm \vphi (x,z)] < 0 ,
\end{equation}
in particular, $\vphi (x) \neq 0$ and $\pa_x^\pm
\vphi (x)  \neq 0 $.
(Recall that $\aone = \aone (\dd M)$ is defined by (\ref{e a2}).)
\end{lem}

Statement (i) is obvious. Statement (ii)  follows immediately from the well-known particular case of the Lagrange identity
$
\im z^2 \int_{0-}^{x\pm} |\vphi (s,z)|^2 \dd M = - \im [\overline{\vphi (x,z)} \pa_x^\pm \vphi (x,z)] .
$
The details can be found in \cite[Section 2]{KK68_II} (note that there is an obvious misprint in
\cite[formula (2.25)]{KK68_II}, namely,
the last part of the equality has to be with minus sign).
From another point of view, (\ref{e im vhi vhi'}) is a reformulation of the fact that
the Titchmarsh-Weyl m-coefficients $\frac{\vphi (x,\sqrt{\la})}{\pa_x^\pm \vphi (x,\sqrt{\la})} $
(associated with problem (\ref{e dMdx eq}), (\ref{e BC0}) on the interval $[0,x)$ and
depending on $\la = z^2$)
are non-degenerate (R)-functions (Nevanlinna functions) in the terminology of \cite{KK68_II}.
One more interpretation of this fact in terms of the trajectory of $\vphi (\cdot,z)$ is given in the next lemma.

\begin{lem} \label{l phi rot}
Let $\dd M $ be a nonnegative measure and $z^2 \not \in \RR$.
Denote by
\[
\text{$\argph \vphi (x) $ the continuous in $x$ branch of $\arg \vphi (x) $ fixed by $\argph \vphi (0) = 0$}.
\]
Then:
\begin{eqnarray}
 \text{in the case } z^2 \in \CC_-, \ \  & \argph \vphi (x) \text{ is strictly increasing on } (\aone,\ell]; \label{e arg phi inc}
 \\
 \text{in the case } z^2 \in \CC_+, \ \  & \argph \vphi (x) \text{ is strictly decreasing on } (\aone,\ell].
 \label{e arg phi dec} \notag
\end{eqnarray}
(Note that $\argph \vphi (x)$ is well-defined due to Lemma \ref{l phi R} and the equality $\vphi (0) =1$).
\end{lem}

The following statement is obvious from the integral equation (\ref{e int ep phi}).

\begin{lem} \label{l phi on int M=0}
Let $ (x_0,x_1) \cap \supp \dd M = \emptyset $. Then the point $\vphi (x) $ with $x$ running through the interval $(x_0,x_1)$
either moves along the ray $\{ \vphi (x_0) + s \pa_x^+ \vphi (x_0) \ : \ s >0  \}$ with the constant speed
$\pa_x^+ \vphi (x_0)$, or stays at $ \vphi (x_0)$ (in the case $\pa_x^+ \vphi (x_0)=0$).
\end{lem}


\subsection{Quasi-eigenvalues' perturbations and the derivatives of $F$.}
\label{ss PEig}

Note that if $\om$ is a quasi-eigenvalue of $\dd M_0$, then
$\om \neq 0 $ (since $F (0; \dd M_0 ) = 1$) and $ \vphi (\ell , \om ; \dd M_0)  \neq 0$
 (otherwise (\ref{e BCl}) implies $\pa_x^+ \vphi (\ell) =0$ and, in turn,
$\vphi (x) \equiv 0$ on $[0,\ell]$).

\begin{prop} \label{p der F}
Let $\om \in \Si (\dd M_0)$ and $\vphi (x)= \vphi (x,\om; \dd M_0)$.
At the point $(\om, \dd M_0 ) \in \CC \times \Mes_\CC$, the derivative in $z$
of the functional $ F (z; \dd M)$ is given by
\begin{equation} \label{e paz F}
\pa_z F (\om; \dd M_0) = -\frac{\vphi
(\ell)}{\om} + \frac{2 \ii}{\vphi
(\ell)}  \int_{0-}^{\ell+} \vphi^2 (s) \ \dd M_0 (s) ,
\end{equation}
and the directional derivative of $F$
w.r.t. $\dd M$ in a direction $\dd V \in \Mes_\CC$ equals
\begin{equation} \label{e padM F}
\frac{ \pa F (\om; \dd M_0)}{\pa M} \ (\dd V) = -\frac{\ii \om}{ \vphi (\ell) }  \int_{0-}^{\ell+} \vphi^2 (s) \ \dd V(s) .
\end{equation}
\end{prop}

\begin{proof}
Put $\psi (x) := \psi  (x,\om; \dd M_0)$.
The formulae
\begin{eqnarray*}
\pa_z F (\om; \dd M_0) = - \frac{\vphi
(\ell)}{\om} + 2 \left[ \om \psi (\ell) + \ii \pa_x \psi (\ell)
\right] \int_{0-}^{\ell+} \vphi^2 (s) \dd M_0 (s) , \\
\frac{ \pa F (\om; \dd M_0)}{\pa M} \ (\dd V) = -\om \left[ \om \psi (\ell) + \ii
\pa_x \psi (\ell) \right] \int_{0-}^{\ell+} \vphi^2 (s) \ \dd V(s) ,
\end{eqnarray*}
can be obtained in the lines of the proof of \cite[Lemma 3.2]{Ka13}
with the change of the usual variation of parameters method to its integral equation version \cite[Sec. 11]{A64} (see also \cite{KK68_II}).
To obtain (\ref{e paz F}) and (\ref{e padM F}), it remains to note that
\begin{equation} \label{e i / phi l}
 \om \psi (\ell) + \ii \pa_x \psi (\ell) = \ii / \vphi (\ell) .
\end{equation}
The last identity easily follows from $\vphi (\ell)
= -\frac{i}{\om} \pa_x^+ \vphi (\ell)$
and the constancy of the Wronskian
\begin{equation} \label{e Wrons}
\left| \begin{array}{cc}
\vphi (x) & \psi (x) \\
\pa_x^\pm \vphi (x) & \pa_x^\pm \psi (x)
\end{array} \right| \equiv 1 \ \  .
\end{equation}
\end{proof}

In the rest of this subsection we assume that $\om $ \emph{is a quasi-eigenvalue of algebraic multiplicity $r \in \NN$
of a nonnegative string} $\dd M_0$. Proposition \ref{p der F} leads to several statements given below.
Though we do not use them directly, they stand behind crucial points of the subsequent sections.

The Fr\'echet derivative $\frac{ \pa F (\om; \dd M_0)}{\pa M}$ is nonzero.
The directional derivative $\frac{ \pa F (\om; \dd M_0)}{\pa M} \ (\dd V)$ is nonzero whenever the direction
$\dd V $ is an atom measure $\delta (x-x_0) \dd x$ with  $x_0 \in \I$ (more generally, whenever nonnegative
and nonzero $\dd V $ is supported on a small enough interval).

\begin{lem} \label{l deg eig}
The quasi-eigenvalue $\om$ is non-simple exactly when
$\vphi^2 (\ell) = 2 \ii \om  \int_{0-}^{\ell+} \vphi^2 (s) \dd M_0 (s) .$
\end{lem}

When $\om$ is a simple quasi-eigenvalue of $\dd M_0$,  the implicit function theorem for analytic maps
implies that there exists a functional $\Omega$ analytic in a certain neighborhood $W$ of $\dd M_0$ such that
$\Omega (\dd M_0) = \om_0$ and
$\Omega (\dd M ) \in \Si (\dd M)$ for all $\dd M \in W$.
The derivative $\frac{ \pa \Omega (\dd M_0)}{\pa M} (\dd V)$ of $\Omega$
at $\dd M_0$ in a direction $\dd V \in \Mes_\CC$ is given by
\begin{equation} \label{e der k}
\frac{ \pa \Omega (\dd M_0)}{\pa M} (\dd V) =
 \frac{\om^2 \int_{0-}^{\ell+} \vphi^2 (s,\om;\dd M_0) \ \dd V(s)}
{ 2 \om \int_{0-}^{\ell+} \vphi^2 (s,\om_0; \dd M_0) \ \dd M_0 (s) +  \ii \vphi^2 (\ell,\om_0; \dd M_0 )} .
\end{equation}

\begin{rem}
For absolutely continuous $\dd M$ and $\dd V$ and the boundary condition $y(0)=0$,
the analogues of Lemma \ref{l deg eig}, formula (\ref{e der k}), as well as higher order corrections for $\Omega$ were obtained
in Physics papers \cite{LLTY94}, \cite{MvdBY01} (with slightly more intuitive arguments).
\end{rem}

When $\om$ is a multiple quasi-eigenvalue (i.e., $r \ge 2$) and $\frac{ \pa F (\om; \dd M_0)}{\pa M} \ (\dd V) \neq 0$,
it is possible to obtain from
the Weierstrass preparation theorem  that
there exist nonempty open discs $\DD_{\delta} (0)$, $\DD_{\ep} (\omega) $,
and a convergent in $\DD_{\delta} (0)$ $r$-valued Puiseux series
\begin{equation} \label{e k Pser}
\Omega (\zeta) = \omega + \sum_{n=1}^{\infty} c_j \zeta^{j/r} \
\text{ with } c_1 = \left[
\frac{ \ii r! \, \om  \ \int_{0-}^{\ell+} \ \vphi^2 (s,\om) \ \dd V(s)}
 {\vphi (\ell,\om) \ \pa_z^r F (\om ; \dd M_0)}
\right]^{1/r}
\neq 0
\end{equation}
such that for each $\zeta \in \DD_{\delta} (0)$, the $r$ values of $\Omega
(\zeta)$ give all the quasi-eigenvalues of $\dd M_0 + \zeta \dd V $ in
$\DD_{\ep} (\omega)$ and all these $r$ quasi-eigenvalues are distinct
and simple (for details, see the proof of \cite[Proposition 3.5]{Ka13} and also Lemma \ref{l P ser} below).
 When $r=1$, (\ref{e k Pser}) turns into (\ref{e der k}).

\section{Local boundary points and two-parameter perturbations}
\label{s 2par pert}

The quasi-eigenvalues of $\dd M$ are zeroes of the entire function $F(\cdot) = F(\cdot; \dd M)$. If $\om$ is a boundary point
of the set of admissible quasi-eigenvalues $\Si [\A]$,
then perturbations of $\om_0$ as a zero of $F(\cdot; \dd M)$ under small changes of $\dd M$ inside the admissible
family $\A$ cannot cover a neighborhood of $\om$. We want to use this fact to derive strong restriction
on the structure of strings that produce quasi-eigenvalues on the boundary of $\Si [\A]$
(and, in particular, on the structure of strings of minimal decay).

This approach requires the study of multi-parameter perturbations of zeroes of analytic maps (functionals) from a Banach space to $\CC$.
In this section, the question is reduced to the two-parameter perturbation theory for
the zeroes of an analytic function of three complex variables.


Let $U$ be a topological space. Consider a functional $G : \CC \times U \to \CC $.
For $u \in U$, denote by
\[
\Si_G (u) := \{ z \in \CC \ : \ G (z;u) = 0 \}
\]
the set of zeroes of the function $G (\cdot, u)$.
Then $\Si_G$ is a set-valued map from $U$ to $\CC$.
We will systematically use only one specific definition of the theory of set-valued maps.
For a subset $S$ of $U$, the image $\Si_G [S]$ of $S$ is defined by
$
\Si_G [S] := \bigcup_{u \in S } \Si_G (u) .
$

\begin{defn} \label{d loc Bd}
Let  $S$ be a set in a topological space $U$. Let $u \in S$ and $z \in \Si_G (u)$. Then:

\item[(i)] the complex number $z$ is called a \emph{$u$-local boundary point of the image} $\Si_G [S]$ if $z$
is a boundary point
of $\Si_G [S \cap W]$ for a certain open neighborhood $W$ of $u$.

\item[(ii)] $z$ is called a \emph{$u$-local interior  point}  $\Si_G [S]$ if $z$ is an interior point
of the set $\Si_G [S \cap W]$ for any open neighborhood $W$ of $u$.
\end{defn}

If $u \in S$, then a complex number $z \in \Si_G (u)$ is a $u$-local interior  point of $\Si_G (S)$
exactly when it is not a $u$-local boundary point of $\Si_G [S]$.
The above definition depends not only on the point $z$ and the set $\Si_G [S]$, but also
on the choice of $u$, $S$, $G$, and a topology in $U$.
Our main (but not only) uses of the above definition will involve $\Mes_\CC$ with the norm or weak* topology as
a topological space $U$ and the functional $F$ defined by (\ref{e F}).
The set-valued map $\dd M \mapsto \Si (\dd M) $ that associates a string $\dd M $ with the set
of the string's quasi-eigenvalues can be identified with the map $\Si_F$, see Lemma \ref{l an} (i).
So $\Si (\dd M) = \Si_F (\dd M)$ and $\Si [\A] = \Si_F [\A]$.

\begin{defn} \label{d loc Bd qe}
Let $\Mes_\CC$ equipped with the norm topology (the weak* topology) be taken as the topological space $U$.
Let $\dd M \in \A$ and $ \om \in \Si (\dd M)$.
Then $\om$ is called \emph{a strongly (resp., weakly*) $\dd M$-local boundary point of}
$\Si [\A]$ if $\om$ is a $\dd M$-local boundary point of $\Si_F [\A]$ in the sense of
Definition \ref{d loc Bd}.
\end{defn}

When $U$ is a Banach space and $G: \CC \times U \to \CC $ is analytic, we denote by
\[
 \frac{\pa G (z;u_0)}{\pa u} (u_1) :=
 \lim_{\zeta \to 0} \frac{G (z;u_0 + \zeta u_1 ) - G (z;u_0)}{\zeta} ,
 \]
\emph{the directional derivative} of $G(z;\cdot)$ along the vector $u_1 \in U$ at the point $(z,u_0) \in \CC \times U$.
By $ \frac{\pa G (z;u_0)}{\pa u} $ the corresponding \emph{Fr\'echet derivative in $u$} is denoted.
The Fr\'echet derivative $\frac{\pa G (z;u_0)}{\pa u} $ is a linear functional on $U$ and
$ \frac{\pa G (z;u_0)}{\pa u} [S]$ defines the image of a set $S \subset U$
under this functional.

The following theorem is one of our main technical tools.

\begin{thm} \label{t ab loc bd}
Let $U$ be a Banach space and $G (\cdot;\cdot)$ be an analytic functional on $\CC \times U$.
Assume that
\item[(i)] $S$ is a convex subset of $U$,
\item[(ii)] $z_0 \in \CC$ and $u_0 \in S$ are such that $G(z_0 , u_0) = 0$,
\item[(iii)] $\frac{\pa G (z_0;u_0)}{\pa u}  (u_0)$ is an interior point of the set of directional derivatives
$\frac{\pa G (z_0;u_0)}{\pa u}  [S]$

Then $z_0$ is a strongly $u_0$-local interior point of $\Si_G [S]$.
\end{thm}

Theorem \ref{t ab loc bd} can be easily obtained from the two following statements.
The first of them gives several equivalent reformulations of the condition (iii) and is an easy exercise in Convex Analysis.

\begin{prop} \label{p iii'iii''}
Under assumptions  (i) and (ii) of Theorem \ref{t ab loc bd},
condition (iii) is equivalent to each of the following conditions:
\item[(iii$\;^\prime$)] The origin $0$ is an interior point of the set $\frac{\pa G (z_0;u_0)}{\pa u}  [S-u_0]$.
\item[(iii$\;^{\prime\prime}$)] $0$ is an interior point of $\cone \frac{\pa G (z_0;u_0)}{\pa u}  [  S-u_0]$.
\item[(iii$\;^{\prime\prime\prime}$)] $\cone \frac{\pa G (z_0;u_0)}{\pa u}  [  S-u_0] = \CC$.

Here and below $\cone W $ is the (nonnegative) convex cone generated by a set $W$.
\end{prop}

The second one, Lemma \ref{l 2par per}, is the technical core of Theorem \ref{t ab loc bd}.
Essentially, it describes local structure of the sets covered by the zeroes of $G(\cdot;u)$
under two-parameter perturbations of $u$.

In Lemma \ref{l 2par per} and its proof, a 2-tuple $\zeta = (\zeta_1,\zeta_2) $ belongs to $ \CC^2$ and $Q(z;\zeta)= Q(z,\zeta_1,\zeta_2)$ is
a function of three complex
variables analytic in a neighborhood of the origin
$\mathbf{0}=(0,0,0)$.  As before, $\Sigma_Q (\zeta) $ is the set of zeroes of $Q (\cdot; \zeta)$ and
$\Sigma_Q [S]  := \bigcup_{\zeta \in S} \Sigma_Q (\zeta) $ for any set $S \subset \CC^2$.
We use (real) triangles
$
T_\de  :=  \{ \zeta= (\zeta_1,\zeta_2) \, : \, \zeta_1,
\zeta_2 \in \RR_+ \ \text{ and } \ \zeta_1 + \zeta_2 < \delta \} .
$

\begin{lem} \label{l 2par per}
Let $Q (z,\zeta_1,\zeta_2) $ be a function of three complex
variables analytic in a neighborhood of the origin
$\mathbf{0}$.
Assume that $0$ is an $r$-fold zero ($1\leq r
< \infty$) of the function $Q (\cdot,0,0)$. Denoting
\begin{equation} \label{e def eta}
\eta_j := -\frac{r!
\pa_{\zeta_{\scriptstyle j}} Q (\mathbf{0})}{\pa_z^r Q
(\mathbf{0})}, \ \ \ j=1,2,
\end{equation}
assume that
\begin{equation} \label{a eta}
\eta_1 \neq 0, \ \ \eta_2 \neq 0 \ \ \text{ and } \ \
 \arg \eta_2 = \arg \eta_1 + \xi_0 \ (\modn 2\pi) \
 \text{ with certain } \ \xi_0 \in (0,\pi).
\end{equation}
Then for small enough and positive $\de_1$ and $\de_2$,
there exists $\ep>0$ such that
\[
\DD_\ep (0) \cap \Sec \left[ \arg \sqrt[r]{\eta_1} + \delta_2 ,  \arg
\sqrt[r]{\eta_2} - \delta_2 \right]
\ \subset \
\Sigma_Q [ T_{\de_1} ]  .
\]
Here $\sqrt[\rr]{\zeta}$ and $\arg z$  are arbitrary branches of the corresponding multi-functions
continuous on the infinite sectors $\Sec [\arg \eta_1 , \arg \eta_1 + \xi_0 ] $ and
$\Sec [\arg \sqrt[r]{\eta_1}  , \arg \sqrt[r]{\eta_1} + \xi_0 / r ] $, respectively.
\end{lem}

The proof of Lemma \ref{l 2par per} is given in Appendix. It is an essential refinement of that of \cite[Lemma 3.6]{Ka13}. With the use of continuous curve of zeroes produced in \cite[Lemma 3.6]{Ka13}, we construct a loop and show that this loop can be  homotopically  shrunken into origin. As a result, the projections of zeroes cover the desired set.

\section{Existence, non-existence, and discrete strings}

The next proposition easily follows from Lemmas \ref{l int},  \ref{l an} (ii), Example \ref{ex 1 pt mas}, and
standard weak* compactness arguments \cite{K51,HS85}
(the scheme of the proof is the same as in \cite{Ka13}).

\begin{prop} \label{p Sclosed}
(i) The set $\Si [\Am] $ is closed.
\item[(ii)] For each frequency $\alpha \in \RR$, there exists a string of minimal decay  over $\Am$.
\end{prop}


In the sequel, we fix a string $\dd M_0 $ and its quasi-eigenvalue $\om$ and then show that various  assumptions on $\om$ impose strong restrictions on $\dd M_0$ and the corresponding mode
\[
  \Phi (x) = \vphi (x,\om; \dd M_0) .
\]
This section is devoted to two following theorems.

\begin{thm} \label{t AM opt}
Let a string $\dd M_0$ be of minimal decay over $\Am$
for a frequency $\alpha \in \RR $.
Then $\dd M_0$ consists of a finite number of point masses, i.e.,
$
\dd M_0 = \sum_{j=1}^n m_j \delta (x-a_j) \dd x
$
with $n \in \NN$,  positive $m_j$, and distinct increasingly ordered $a_j \in \I$.
\end{thm}

\begin{rem}
Theorem \ref{t AM opt} will be further refined in Sections \ref{ss res weak*} and \ref{s c rem}.1,
where additional relations connecting $a_j$, $m_j$, and $\Phi$ will be derived and an iterative procedure for calculation of $m_j$ and $a_j$ will be given.
\end{rem}

The proof of the theorem is given in Subsection \ref{ss proofs opt}.
Essentially, it is based on the study of $\dd M_0 \in \Am$ and $\om \in \Si [\Am] $
such that
\begin{eqnarray}
& \om \text{ is a strongly $\dd M_0$-local boundary point of } \Si [\Am] ,
\label{e om loc-bd}
\end{eqnarray}
see Subsection \ref{ss strong loc-b}.
These subsection  contains a description of specific properties of the corresponding mode
$\Phi$
and their connections with the structure of $\dd M_0$.
Then, Theorem \ref{t AM opt} is obtained with the use of the obvious implications
\begin{eqnarray}
\text{$\dd M_0$ and $\om$} & \text{ are of } & \text{minimal decay for a frequency $\alpha $ (over $\A$)} \label{e impl loc bd md}
\\
& \Longrightarrow &
 \om \text{ is a weakly* $\dd M_0$-local boundary point of $\Si [\A]$} \label{e impl loc bd weak} \\
& \ & \Longrightarrow  \om \text{ is a strongly $\dd M_0$-local boundary point of $\Si [\A]$} . \label{e impl loc bd strong}
\end{eqnarray}

As a by-product of the study of optimization over $\Am$, we obtain the following result concerning optimization over
the admissible family $\A_1$.

\begin{thm} \label{t A1}
There are no strings of minimal decay over $\A_1$ (for any frequency $\alpha \in \RR$).
\end{thm}

This result is proved in Subsection \ref{ss proofs opt}.

Recall that $ \aone = \aone (\dd M_0)$ equals \ $\min \; \supp \dd M$ when $\dd M$ is nonzero, otherwise $\aone := \ell$.
The trajectories associated with the functions $\Phi (\cdot)$ and
$\Phi^2 (\cdot) := (\Phi (\cdot))^2$ are denoted by
\[
\Phi [\I] := \left\{  \  \Phi (x) \ : \  x  \in [0,\ell] \right\}  \text{ and }
\Phi^2 [\I] :=  \left\{  \  \Phi^2 (x)  \ : \  x  \in [0,\ell] \right\} .
\]


In the sequel, points of $\CC$ are perceived both
as complex numbers and as $\RR^2$-vectors.
In particular, \emph{hyperplanes in $\CC$ are lines}. By $\langle z_1 , z_2 \rangle_\CC := \re z_1 \re z_2 + \im z_1 \im z_2 $ we denote
the $\RR^2$ scalar product of two complex numbers.
Note that, for every $c \in \CC$,
\begin{equation} \label{e c<zz>}
\langle cz_1 , cz_2 \rangle_\CC = |c|^2 \ \langle z_1 , z_2 \rangle_\CC .
\end{equation}

In this and subsequent sections, by $\sqrt{z}$ and $\argb z$ we denote continuous in $\CC \setminus \RR_-$ branches of $z^{1/2}$ and $\arg z$, resp.,
fixed by $\sqrt{1} =1$ and $\argb 1 = 0$. For $z \in \RR_- $, put $\argb z = - \pi$.

\subsection{Lemmas on directions producing extremal derivatives.}
\label{ss yBone}

Denote
\[
\Bone := \{ \dd M \in \Mes_+ \ : \ \| \dd M \| \le 1 \}
\ \text{ and } \ \Bac := \{ \dd M \in \Bone \ : \ \dd M = B(x) \dd x \text{ with } B \in L^1 \} .
\]
Since $\Am = m \Bone$,
Proposition \ref{p der F} shows that
the set of the directional derivatives
$\frac{ \pa F (\om; \dd M_0)}{\pa M} \ (\dd V) $ with $\dd V$ running through the admissible family $\Am$
is the image of  $\Bone$ under the functional
$ \left<  - m \frac{\ii  \om}{ \Phi (\ell)} \Phi^2 , \, \cdot \, \right> $ (defined on $\Mes$). Boundary points of this image will play a special role due to Theorem \ref{t ab loc bd}.

With a  function $y \in C [0,\ell]$ and a set $W \subset \Mes$, we associate the set of complex numbers
\begin{equation*} \label{e yBone}
\langle y , W \rangle =
\left\{ \int_{0-}^{\ell+} y (s) \ \dd V(s) \ : \ \dd V \in W \right\} .
\end{equation*}
Note that $\langle y , W \rangle$ is convex whenever $W$ is convex and, in particular,
when $W = \Bone$ or $W = \Bac$.

\begin{lem} \label{l zeta extr}
Let $W = \Bone$ or $W = \Bac$. Suppose  $\dd V \in W$, $\zeta = \langle y , \dd V \rangle$,
and that $\zeta $ is an extreme point of $\langle y , W \rangle$.
Then $y [\supp \dd V] := \{ y (x) \ : \ x \in \supp \dd V \}$ is a subset of $\{ \zeta \}$, i.e., $y [\supp \dd V]$ either consists of
the single point $ \zeta $, or is empty. If, additionally, $\|\dd V \| < 1$, then $\zeta = 0$.
\end{lem}

\begin{proof}
\emph{Step 1. Let us prove the lemma for the case when $y$ is real-valued.} Obviously,
\[
\left( \min\limits_{x \in \I} y (x), \max\limits_{x \in \I} y (x) \right)  \cup  \{ 0 \} \  \subset \  \langle y , W \rangle
\ \subset \ \conv \{ \min\limits_{x \in \I} y (x) , \max\limits_{x \in \I} y (x) , 0 \} .
\]

Assume that $ \min y \cdot \max y \le 0$. Since $\zeta$ is extreme, we see that either $\zeta = \min y $, or $\zeta = \max y$. To be specific, assume $\zeta = \max y$.
If $\| \pr_{\{x : y(x) < \zeta \}} \dd V \| >0$, then
$ \zeta = \int y \dd V < \zeta $, a contradiction.
So $y(x) =\zeta $ for all $x \in \supp \dd V$, and
$ \zeta = \zeta \| \dd V \|$.
The latter implies that $\zeta =0$ in the case $\|\dd V \| < 1$.

These arguments work also in the case $\zeta \neq 0$. Assume now that $ \min y \cdot \max y > 0$
and $\zeta = 0$. Then $\dd V = 0 \dd x$ and $\supp \dd V = \emptyset$.

\emph{Step 2. The general case.}
Since $\zeta$ is extreme, there exists a supporting line $L$
to $\langle y , W \rangle$ at $\zeta$. Let us write $L$ in the form
$\{ z \in \CC : \ \langle z -\zeta, p \rangle_\CC \ = 0\}$ with certain $p \in \CC \setminus \{ 0 \}$.
Applying Step 1 to the real-valued function
$y_1 (x) := \langle y(x) , p \rangle_\CC $, one can show that
$y [\supp \dd V]  \subset L $ and, moreover, that $0 \in L$ in the case $\|\dd V \| < 1$.

Since $\zeta $ is an extreme point of $S_1 := L \cap \langle y , W \rangle$,
we see that
\begin{equation}
\text{$S_1 $ lies in one of the two rays $\zeta \pm \ii p \Rnneg$.} \label{e S1 in rays}
\end{equation}
In the case $\| \dd V \|=1$, this immediately implies that $\| \pr_{\{y(x) \neq  \zeta \}} \dd V \| = 0$
and $y [\supp \dd V]  = \{ \zeta \}$.

In the case $\| \dd V \|<1$, we see that the line-segment $[0, \| \dd V \|^{-1} \zeta ]$ is a subset of $S_1$.
Since $\zeta$ is extreme, $\zeta = 0$.
Now (\ref{e S1 in rays}) easily implies $y [\supp \dd V]  \subset \{ \zeta \}$.
\end{proof}

It is essential for the rest of the paper that $\langle y , \Bone \rangle$ contains
the trajectory $y [\I] = \{ y (x) : x \in [0,\ell] \}$ and the origin $0$.
The following sharpening of this obvious fact is not crucial, but makes many of subsequent proofs more transparent.

\begin{lem} \label{l conv Trj y}
$ \langle y , \Bone \rangle  = \conv \left( \, y [\I ]  \cup \{0\} \, \right) $.
\end{lem}

\begin{proof}
By the Helly selection and convergence theorems, the set $ \langle y , \Bone \rangle $ is compact and convex. So, by the Minkowski theorem, $\langle y , \Bone \rangle $
is a convex hull of the set $\Ext$ of its extreme points. Let us show that $\Ext \subset y [\I ] \cup \{0\}$. Indeed,
assume that a nonzero point $z$ is extreme and $ z = \langle y , \dd V \rangle$ with $\dd V \in \Bone$.
By Lemma \ref{l zeta extr}, $\| \dd V\|=1$ and $y [\supp \dd V] = \{z\}$. Taking
arbitrary $x_0 \in \supp \dd V$ one can see that
$
 z =  \langle y , \de (x - x_0) \dd x \rangle = y (x_0) \in y [\I ] .
$
\end{proof}


\begin{lem} \label{l supp line}
Let $W = \Bone$ or $W = \Bac$. Assume that $\dd V \in W $ and $\zeta := \langle y , \dd V \rangle$
is a boundary point of the set $ \langle y , W \rangle $.
Then there exists a supporting line (hyperplane) $L$ to the set $ \langle y , W \rangle $ at the point $ \zeta $.
For every such $L$ the following assertions hold:
\item[(i)] $y [\supp \dd V] \subset L$,
\item[(ii)] if additionally $\| \dd V \| < 1$, then $0 \in L$.
\end{lem}

\begin{proof}
The existence of a supporting line follows immediately from $\zeta \in \Bd \langle y , W \rangle $.
Assertions (i) and (ii) can be easily obtained from the real-valued case of Lemma \ref{l zeta extr} and the arguments of Step 2 of its proof.
\end{proof}

Now we apply this lemma to  strongly local boundary points.
Assume that $\A = \Am$ or $\A = \A_1$ and $\dd M_0 \in \A$.
The set $\langle \Phi^2 , m^{-1} \A \rangle$
consists of the complex numbers
$\langle \Phi^2 , m^{-1} \dd M \rangle := \frac1m \int_{0-}^{\ell+} \Phi^2 (s) \ \dd M(s) $ produced by admissible strings $\dd M \in \A$.
On the other side, by (\ref{e padM F}),
\begin{equation} \label{e C0}
\langle \Phi^2 , m^{-1} \dd M \rangle  =  \frac {C_0}m  \ \frac{ \pa F (\om, \dd M_0) }{\pa M} \ (\dd M) , \ \
\text{ where } C_0 := - \frac{\Phi (\ell)}{\ii \om} \text{ is a nonzero constant}.
\end{equation}
So $\langle \Phi^2 , m^{-1} \A \rangle$ is the image of the admissible family $\A$ under
the Fr\'echet derivative $\frac{ \pa F (\om, \dd M_0) }{\pa M}$ additionally scaled and rotated
by the multiplication on $C_0/m$.

The point
\begin{equation}
\text{$z_0 :=  \langle \Phi^2 , m^{-1} \dd M_0 \rangle $ can be written as} \  z_0 =  \frac {C_0}m \ \frac{ \pa F (\om, \dd M_0) }{\pa M} \ ( \dd M_0 ) . \label{e z0}
\end{equation}
It belongs to $\langle \Phi^2 , m^{-1} \A \rangle$ and plays a special role due to Theorem \ref{t ab loc bd}.

\begin{prop} \label{p z0}
Let $\A = \Am$ or $\A= \A_1$.
Let $\dd M_0 \in \A$ and let $\om$ be a strongly $\dd M_0$-local boundary point of
$ \Si [\A] $. Then
$z_0 $ is a boundary point of $\langle \Phi^2 , m^{-1} \A \rangle$, and for every supporting line $L$ to the set $\langle \Phi^2 , m^{-1} \A \rangle$
at $ z_0 $,  the following assertions hold:
\item[(i)] $L$ contains the image $\Phi^2 [\supp \dd M_0] $,
\item[(ii)] $1 \in L $, i.e., the line passes through the point $z=1$,
\item[(iii)] if additionally $\| \dd M_0 \| < m$, then $ L = \RR$.
\end{prop}

\begin{proof}
The assertion that $z_0 $ is a boundary point of $\langle \Phi^2 , m^{-1} \A \rangle$ follows from Theorem \ref{t ab loc bd} applied to the functional $F$ and the set $\A$.
Statement (i) follows from Lemma \ref{l supp line}.
To prove (ii), note  that $\aone \in \supp \dd M_0$ and $\Phi (\aone) = 1$.
Now (iii) follows from  (ii) and Lemma \ref{l supp line}.
\end{proof}

\subsection{Strings producing strongly local boundary points}
\label{ss strong loc-b}

In this subsection, we assume that  $\om$ is a strongly $\dd M_0$-local boundary point of $ \Si [\Am] $
and  $\re \om \neq 0$.
Without narrowing generality, the last assumption  can be replaced by
$
\re \om > 0 .
$

\begin{rem} \label{r sym}
Indeed,  $\Si [\Am]$ is symmetric w.r.t. $\ii \RR$, as well as $\Si (\dd M)$ for $\dd M \in \Mes_+$.
Thus, if $\om=\alpha - \ii \beta$ is a strongly (weakly*) $\dd M_0$-local boundary point of $\Si [\Am]$ then
$\overline{-\om}= -\alpha - \ii \beta$ is so.
Note that according to Definition \ref{d loc Bd qe}, assumption (\ref{e om loc-bd})    yields $\om \in \Si(\dd M_0)$.
\end{rem}

Recalling that $\Si [\Am ] \subset \CC_- $, we see that $\om$   belongs to the quadrant $\QQ_{\mathrm{IV}}$.
So
\begin{equation} \label{e xi0}
\xi_0 := \argb (-\om^2)  \text{ belongs to } (0,\pi).
\end{equation}

Under these assumptions the measure $\dd M_0$ is nonzero and
\begin{equation} \label{e aone<l}
\text{$\aone = \min \ \supp \dd M_0 < \ell$.}
\end{equation}
This follows from Examples \ref{ex empty}, \ref{ex 1 pt mas}, and the fact that $\dd M_0$ produces $\om \not \in \ii \RR$.


With $\dd M_0 $ and $\om$, we associate the set of complex numbers
\begin{equation} \label{e S0}
S_0 := \langle \Phi^2 , \Bone \rangle =
\left\{ \int_{0-}^{\ell+} \Phi^2 (s) \ \dd V(s) \ : \ \dd V \in \Mes_+ \text{ and  } \| \dd V \| \le 1 \right\} .
\end{equation}
The set $S_0$ is a convex and, by Lemma \ref{l conv Trj y},
\begin{equation} \label{e conv Trj 2}
S_0 = \conv \left( \ \Phi^2 [\I]  \cup \{0\} \ \right).
\end{equation}
Obviously, the point $z_0$ defined by (\ref{e z0}) belongs to $S_0$.

\subsubsection{Hyperbolic "billiard"}
\label{sss HypBill}

Assume additionally that
\begin{eqnarray}
\text{there exists a supporting line $L$ to $ S_0 $ at $ z_0 $ such that } \ \ 0 \not \in L .
 \label{e 0 nin L}
\end{eqnarray}

We introduce the following parametrization of $L$. Let $p_0$ be the point of $L$ closest to $0$.
Since $p_0 \neq 0$, we see that $p := \frac{p_0}{|p_0|}$ is a  normal unit vector to $L$.
Proposition \ref{p z0} (ii) yields $1 \in L$. So
$
L = 1+ \ii p \RR = \{ 1 + \ii p s \ : \ s \in \RR \} .
$
The line $L$ divides $\CC \setminus L$ into two open half-planes
\begin{equation}
H_0 = H_0 (p) := \{ z \in \CC \ : \ \langle z- 1 , p \rangle_\CC < 0 \} \ \ \ \
\text{ and \ \ $\CC \setminus \overline{H_0}$.}
\end{equation}
 Note that $ 0 \in H_0 $ \ since
\ \ $\re p = \langle 1 , p \rangle_\CC > 0$ \ \ and \ \
$ H_0 = \{ z \in \CC \ : \ \langle z , p \rangle_\CC < \re p \}$.

The preimage of $L$ under the map $z \mapsto z^2$ is a rectangular hyperbola $\Hyp$ consisting of two branches
$\Hyp^+$ and $\Hyp^-$,
\begin{equation} \label{e Hyp pm}
\Hyp^\pm  := \{ \pm  \sqrt{1+\ii p s } \ : \
s \in \RR \} .
\end{equation}
Clearly, $\pm 1 \in \Hyp^\pm$.
The hyperbola $\Hyp$ divides $\CC \setminus \Hyp$ into three connected components. By $\Cmp_0$ we denote
the connected component containing $0$,
$ \
\Cmp_0 = \{ \zeta \in \CC  \ : \ \langle \zeta^2 , p \rangle_\CC < \re p \} .
$
The two other components
\begin{equation} \label{e Hcmp pm conv}
\Cmp_\pm = \{ \zeta \in \CC \ : \ \pm\re  \zeta > 0 \text{ and } \langle \zeta^2 , p \rangle_\CC > \re p \} \text{ are convex} .
\end{equation}

Since $L$ is a supporting line and $0 \in S_0$, we see that
\begin{equation} \label{e S0 sub H0}
S_0 \subset \overline{H_0} .
\end{equation}
This, (\ref{e conv Trj 2}), and the fact that $\Cmp_0$ is the preimage of $H_0$ under $z \mapsto z^2$
imply
\begin{equation} \label{e Trj in}
\Phi^2 [\I]  \subset \overline{H_0} \ \ \ \text{ and } \ \Phi [\I] \subset \overline{\Cmp_0} .
\end{equation}

Differentiating parameterizations (\ref{e Hyp pm}), we see that, for
the point $\zeta = \pm  \sqrt{1+\ii p s} \in \Hyp$,
\begin{eqnarray}
\frac{\pm \ii p}{\sqrt{1+\ii p s}} &  \ \text{is a tangent vector to $\Hyp$
at $\zeta $,} \ \ \  \notag \\
\text{  and } \ \frac{\pm p}{\sqrt{1+\ii p s}} & \ \text{is a normal to $\Hyp$
at $\zeta $ pointing toward $\Cmp_\pm$}. \label{e nv Hyp}
\end{eqnarray}


\begin{prop} \label{p billiard}
Suppose $\re \om >0$, (\ref{e om loc-bd}), and (\ref{e 0 nin L}).
Then $\| \dd M_0 \| = m $ and there exists a finite set of real numbers $\{ a_j \}_{j=1}^{n+1}$
with the following properties:
\item[(i)] $ \aone = a_1 , \ \ a_j < a_{j+1} \le \ell \text{ for } 1 \le j \le n-1, \  \
\text{ and } a_{n+1} = \ell $,
\item[(ii)] $ \supp \dd M_0 = \{ a_j \}_{j=1}^{n} $,
\item[(iii)] The trajectory $\Phi [\I]$ is a piecewise linear path consisting of the
(closed) line-segments \linebreak $[\Phi (a_j) , \Phi(a_{j+1})]$, $j=1, \dots, n$, lying in the closure
$\overline{\Cmp_0}$.
\item[(iv)] For $1 \le j \le n$, the endpoints $\Phi (a_j)$ of these line-segments lies on the hyperbola $\Hyp$. Moreover,
\[
\Phi (a_j) \in \Hyp^+ \text{ if $j$ is odd , \ \ \ \ and } \
\Phi (a_j) \in \Hyp^- \text{ if $j$ is even} .
\]
\item[(v)] If $n \ge 2$ and real numbers  $ s_j $ \ $(j=1, \dots, n)$  are such that
$\Phi^2 (a_j) = 1+\ii p s_j$, then
\begin{eqnarray}
  s_1  =  0 ; \qquad   s_j  >  s_{j+1} \ \text{ and  } \ s_j \, > \, \frac{\left< \om^2 , p \right>_\CC}{\im \om^2} \ \ \text{ for } 1 \le j \le n-1 .
  \label{e sj>sj+1}
 \end{eqnarray}
\item[(vi)] If $a_n < \ell$, then $s_n \ge \frac{\left< \om^2 , p \right>_\CC}{\im \om^2}$.

(Note that (i) allows $a_n=a_{n+1} = \ell$.
According to (i)-(ii), $a_n=a_{n+1} = \ell $ \quad $\Leftrightarrow $ \quad $\ell \in \supp \dd M_0$).
\end{prop}

\subsubsection{The proof of Proposition \ref{p billiard}}
\label{sss proof billiard}

\begin{lem} \label{l aone 0ninL new}
(i) \ \ $\langle \om^2 , p \rangle_\CC \ge 0$
\item[(ii)] \ \ $a_* $ is an isolated point of $\supp \dd M_0$.
\end{lem}

\begin{proof}
Recall that $\xi_0 = \argb (-\om^2)$ and, by Lemma \ref{l phi rot} and (\ref{e xi0}),
\begin{equation} \label{e Phi inc}
 \argph \Phi (x) \text{ is strictly increasing on } [\aone,\ell] .
\end{equation}

For small enough $\ep>0$, there exists $\de = \de(\ep) > 0$
such that
$|\Phi (x) - 1|<\ep$ and $\Phi (x) \in \Sec [0,\ep]$ for all $x \in (\aone,\aone+\de]$.
Combining this and $\Phi (\aone) =1$ with the integral equation (\ref{e int ep phi})
($\Phi$ satisfies (\ref{e int ep phi}) with $z=\om$),
it is easy to see that
\begin{equation} \label{e Phi-1}
\Phi (x) - 1 \ \in \ \Sec [\xi_0 , \xi_0 + \ep] \cap \DD_\ep (0) \ \text{  for all } x \in (\aone,\aone+\de] .
\end{equation}

Let us prove \textbf{(i)} by \emph{reductio ad absurdum}. Assume
$\langle -\om^2 , p \rangle_\CC > 0$. Note that $p$ is a normal to $\Hyp$
at $1$ pointing toward $\Cmp_+$, see e.g. (\ref{e nv Hyp}).
This, $\langle -\om^2 , p \rangle_\CC > 0$, and (\ref{e Phi-1}) imply that, for small enough $\ep>0$,
\[
\text{$\Phi (x) \in \Cmp_+$ \ \ \ for all $x \in (\aone,\aone+\de(\ep)]$.}
\]
The latter contradicts (\ref{e Trj in}) since $\aone < \ell$.


\textbf{(ii)} Since $\re p >0$ and $\im (-\om^2) >0$, we see that
$\xi := \xi_0 - \argb p \in (-\pi/2, 3\pi/2)$.
Statement (i) implies $\xi \in [\pi/2, 3\pi /2)$. Combining this, (\ref{e Phi-1}), and taking $\ep$ small enough
we see that
\[
\text{$0<|\Phi (x) - 1| < \ep$ \ \ and \ \  $\langle \Phi (x) -1 , p \rangle_\CC \le 0$ \ \ for $x \in (\aone , \aone+ \de (\ep))$}.
\]
For sufficiently small $\ep$, the set
$ \{ z \ : \ 0<|z-1| < \ep  \ , \  \langle z -1 , p \rangle_\CC \le 0 \} $ lies in $ \Cmp_0 $.

Taking such $\ep$, we ensure $\Phi(x) \in \Cmp_0$. So $\Phi^2 (x) \not \in L$ for all
$(\aone , \aone+ \de (\ep))$. By Proposition \ref{p z0}, $\supp \dd M_0 \cap (\aone , \aone+ \de (\ep)) = \emptyset$.
This completes the proof.
\end{proof}

Let $m_{\{x\}}$ be the mass of the point $x$ w.r.t. the measure $\dd M_0$.

\begin{lem} \label{l aj}
Assume that $x_0, x_1 \in \supp \dd M_0$ and $t_0 \in \RR$ are such that
\begin{eqnarray}
 x_0 < x_1 , \quad m_{\{x_0\}} & > & 0,  \ \ \ \ \
(x_0,x_1) \cap \supp \dd M_0 = \emptyset ,
\label{e x0x1 notin supp}
\\
\text{ and } \quad \Phi (x_0) & = & \pm  \sqrt{1+ \ii p t_0} \ \ \ \ \
\text{(in particular, $\Phi (x_0) \in \Hyp^\pm$)} . \label{e Phi x0}
\end{eqnarray}

Then:
\item[(i)] $\{\Phi (x) \ : \ x \in [x_0,x_1] \}$ is a non-degenerate closed line-segment
$[\Phi (x_0), \Phi (x_1)]$. Moreover,
\begin{equation} \label{e vx0x1}
\text{ for all $x \in (x_0,x_1)$, \ \ \ $\pa_x \Phi (x) = \pa_x^+ \Phi (x_0) = \pa_x^- \Phi (x_1) =: v_{(x_0,x_1)}$.}
\end{equation}
\item[(ii)] There exists $t_1 \in \RR$ such that $\Phi (x_1) = \mp  \sqrt{1+\ii p t_1}$
(this means that $\Phi (x_1) \in \Hyp^-$ if $\Phi (x_0) \in \Hyp^+$, and vise versa).
\item[(iii)] $\bigl( \Phi (x_0), \Phi (x_1) \bigr) \subset \Cmp_0$
(recall that $(z_1,z_2) := [z_1,z_2] \setminus \{z_1,z_2\}$).
\item[(iv)] $\left< \pa_x^- \Phi (x_0) , \frac{\pm p}{\sqrt{1+\ii p t_0}} \right>_\CC  \ge 0  $
and $\left< \pa_x^- \Phi (x_1) , \frac{\mp p}{\sqrt{1+\ii p t_1}} \right>_\CC  > 0  $.
\item[(v)] $t_0> \frac{\left< \om^2 , p \right>_\CC}{\im \om^2} $ and $t_0 > t_1$
\item[(vi)] If additionally $x_1 < \ell$, then $t_1 > \frac{\left< \om^2 , p \right>_\CC}{\im \om^2} $.
\item[(vii)] $x_1$ is an isolated point of $\supp \dd M_0$ (in particular, $m_{\{x_1\}} >0$).
\item[(viii)] $x_1 - x_0 > \frac{2|p_0|^{1/2}}{| v_{(x_0,x_1)} |}
\ge \frac{2|p_0|^{1/2}}{\| \pa_x \Phi \|_{L^\infty} }$

(Recall that $\pa_x \Phi (x)$ exists for a.a. $x \in [0,\ell]$ and is in $L^\infty$ due to
(\ref{e pa phi 2}) and Lemma \ref{l an} (ii).)
\end{lem}

\begin{proof}
\textbf{(i)} follows from Lemmas \ref{l phi on int M=0} and \ref{l phi R} .

\textbf{(ii)} Since $x_0, x_1 \in \supp \dd M_0$, Proposition \ref{p z0} yields $\Phi (x_0), \Phi (x_1) \in \Hyp$.
%
%
Assume that $\Phi (x_0)$ and $\Phi (x_1)$ lie on the same branch of $\Hyp$, say $\Hyp^+$. Then
the line-segment $\bigl( \Phi (x_0), \Phi (x_1) \bigr)$  belongs to
$\Cmp_+$. This follows from (\ref{e Hcmp pm conv}) and the fact that $\Hyp_\pm$
do not contain line-segments. However, $\bigl(\Phi (x_0), \Phi (x_1) \bigr) \subset \overline{\Cmp_0}$
according to (\ref{e Trj in}), a contradiction.

\textbf{(iii)} Due to (\ref{e Trj in}), it is enough to prove that $\Phi (x) \not \in \Hyp$ for all $x \in (x_0, x_1)$.
The latter easily follows from the arguments of the previous step.

\textbf{(iv)} By (\ref{e BC0}),
$ \left< \pa_x^- \Phi (x_0) , \frac{\pm p}{\sqrt{1+\ii p t_0}} \right>_\CC  = 0 $ if $x_0 = 0 $.
Assume $x_0 > 0$ and $\left< \pa_x^- \Phi (x_0) , \frac{\pm p}{\sqrt{1+\ii p t_0}} \right>_\CC <0$.
This and (\ref{e nv Hyp}) yield that $\Phi (x) \in \Cmp_\pm$ for $x<x_0$. The latter contradicts (\ref{e Trj in}).

The same arguments show that $\left< \pa_x^- \Phi (x_1) , \frac{\mp p}{\sqrt{1+\ii p t_1}} \right>_\CC \ \ge 0  $.
Assume that $\left< \pa_x^- \Phi (x_1) , \frac{\mp p}{\sqrt{1+\ii p t_1}} \right>_\CC = 0  $. Then statement (i)
implies that the line-segment $[\Phi (x_0), \Phi (x_1)]$ lies on the tangent line $L_1$ to $\Hyp$ at the point $\Phi (x_1)$.
The intersection of the tangent line $L_1$ and $\Hyp$ consists of exactly one point. This contradicts (i).
Thus, $\left< \pa_x^- \Phi (x_1) , \frac{\mp p}{\sqrt{1+\ii p t_1}} \right>_\CC  > 0  $.

\textbf{(v)} By (\ref{e pa phi 1}),
\begin{equation} \label{e pa+ = pa- -om}
\pa_x^+ \Phi (x) = \pa_x^- \Phi (x) - \om^2 \Phi (x)  m_{\{x\}} .
\end{equation}
Let us show that
\begin{equation} \label{e -om norm}
\left<  -\om^2 \Phi (x_0), \frac{\pm p}{\sqrt{1+\ii p t_0}} \right>_\CC < 0.
\end{equation}

 Assume converse.
Then (\ref{e pa+ = pa- -om}), $m_{\{x_0\}} > 0$, and statement (iv) imply
$\left< \pa_x^+ \Phi (x_0) , \frac{\pm p}{\sqrt{1+\ii p t_0}} \right>_\CC \, \ge 0$.
The sign $>$ in the last inequality  leads to a contradiction due to (\ref{e nv Hyp})
and  (\ref{e Trj in}), see the proof of (iv).
So $\left< \pa_x^+ \Phi (x_0) , \frac{\pm \sqrt{p}}{\sqrt{1+\ii p t_0}} \right>_\CC \, = 0$
and, by (i), the line-segment $[\Phi (x_0) , \Phi (x_1)]$ lies on the tangent line to $\Hyp$ at $\Phi (x_0)$.
This contradicts the fact that $\Phi (x_1)$ belongs to the other branch of $\Hyp$, see the end of the proof of (iv).
Thus, (\ref{e -om norm}) is proved.

Rewriting (\ref{e -om norm}) with the use of (\ref{e Phi x0}),
we get $0 < \left< \om^2 (\pm 1) \sqrt{1+\ii p t_0} \ , \ \frac{\pm p}{\sqrt{1+\ii p t_0}} \right>_\CC $.
Using (\ref{e c<zz>}) and $|p|=1$, let us modify this inequality:
\begin{equation}
0 < \left< \om^2 (1+\ii p t_0) , p \right>_\CC = \left< \om^2 , p \right>_\CC + t_0 |p|^2 \left<  \ii \om^2, 1 \right>_\CC
= \left< \om^2 , p \right>_\CC + t_0 \im (- \om^2) .
\end{equation}
This yields $t_0 > \frac{\left< \om^2 , p \right>_\CC}{\im \om^2}$.

The assertion $t_0 > t_1$ is geometrically obvious if one takes (\ref{e Phi inc}) and the central symmetry of $\Hyp$
into account.

\textbf{(vi)} Using (\ref{e pa+ = pa- -om}) and
$\left< \pa_x^- \Phi (x_1) , \frac{\mp p}{\sqrt{1+\ii p t_1}} \right>_\CC \ > 0  $ from statement (iv),
one can apply the arguments of the previous step to prove
\begin{equation} \label{e -om norm Ph x1}
\left<  -\om^2 \Phi (x_1), \frac{\mp p}{\sqrt{1+\ii p t_1}} \right>_\CC <  0 ,
\end{equation}
and then $t_1 > \frac{\left< \om^2 , p \right>_\CC}{\im \om^2} $.

\textbf{(vii)} If $x_1 = \ell$, the statement is obvious from (\ref{e x0x1 notin supp}).

Assume $x_1 < \ell$.
Adjusting the proof of Lemma \ref{l aone 0ninL new} (ii), it is possible to show that
\begin{equation} \label{e Phx>x1}
\text{$\Phi^2 (x) \not \in L$ for sufficiently small positive $x - x_1 $}.
\end{equation}
Then Proposition \ref{p z0} gives the desired result.

\textbf{(viii)} The distance between $\Hyp_+$ and $\Hyp_-$ equals $2 |p_0|^{1/2}$.
So statement (i) implies (viii).
\end{proof}

\begin{lem} \label{l an l}
Let $x_0 = \max \ \supp \dd M_0$ and assume $x_0 < \ell$. 
Put $x_1 = \ell$.
Then:
\item[(i)] statement (i) of Lemma \ref{l aj} is valid for the part $\{\Phi (x) \ : \ x \in [x_0,\ell] \}$
of the trajectory $\Phi [\I]$,
\item[(ii)] $\left(\Phi (x_0), \Phi (\ell) \right) \subset \Cmp_0$.
\end{lem}

The lemma can be easily obtained by the arguments of the proof of Lemma \ref{l aj}.

\textbf{Now we are able to prove Proposition \ref{p billiard}}.
The equality $\| \dd M_0 \| = m $ follows immediately from $0 \not \in L$ and Proposition \ref{p z0} (iii).

\emph{Consider the case $\supp \dd M_0 = \{ \aone \}$.} Putting $n=1$, $a_1 = \aone$, and $a_2 = \ell$,
we  ensure that (i)-(iii) follows from Lemma \ref{l an l}.  Note that statements (iv) and (v)
involve only $a_1$ and $s_1$. This makes them trivial.
Lemma \ref{l aone 0ninL new} (i) and $\im \om^2 < 0$ yield $s_1 = 0 \ge \frac{\left< \om^2 , p \right>_\CC}{\im \om^2} $.
This gives (vi).

\emph{Consider the case $\supp \dd M_0 \setminus \{ \aone \} \neq \emptyset$.} In this case,
there exists $a_2 := \inf \left( \supp \dd M_0 \setminus \{ a_1 \} \right)$, where $a_1 := \aone$ as before.
By Lemma \ref{l aone 0ninL new} (ii), $a_2 > a_1$. So Lemma \ref{l an l} is applicable to the interval $[a_1,a_2]$
and yields that $a_2$ is an isolated point of $\supp \dd M_0$. If $a_2 = \max \ \supp \dd M_0$, we apply
Lemma \ref{l an l}. Otherwise, there exists $a_3 := \min \left( \supp \dd M_0 \setminus \{ a_j \}_{j=1}^2 \right)$
and $a_3 > a_2$. Continuing this process, we obtain a set of point $\{ a_j \}_{j=1}^n$, $n \le \infty$,
such that the support of $\dd M_0$ in the interval $[0, \sup_j a_j ]$ consists of the points
$a_j$ and $\sup a_j$. Lemma \ref{l aj} (viii) implies $a_{j+1} - a_j > \frac{2|p|^{1/2}}{\| \pa_x \Phi \|_{L^\infty} }$.
So the lengths of the intervals $[a_j,a_{j+1}]$ (with $j < n$) are separated from $0$. Thus, $n < \infty$ and this immediately yields $\supp \dd M_0 = \{ a_j \}_{j=1}^n$.
The rest of statements of the proposition easily follows from Lemmas \ref{l aj} and \ref{l an l}.
This completes the proof.

\subsubsection{The case when the hyperbola degenerates}
\label{sss deg case}

Assume now that
\begin{eqnarray}
& \text{the real line $\RR$ is a supporting line to the set $ S_0 $ at $ z_0 $} .
 \label{e 0 in L}
\end{eqnarray}
Proposition \ref{p z0} implies that (\ref{e 0 in L}) is fulfilled whenever $\| \dd M_0 \| < m$.

From the point of view of the settings of Section \ref{sss HypBill}, the assumption (\ref{e 0 in L})
means that
\[
\text{ $p = -\ii$, \ \ \ \ \ \ $H_0 = H_0 (-\ii) = \CC_+$, }
\]
and the hyperbola $\Hyp$
degenerates into the union of coordinate axes $\RR \cup \ii \RR$.
The set $\overline{\Cmp_0}$ degenerates into the union $\overline{\QQ_I} \cup \overline{\QQ_{III}} $ of the closures of the first and the third quadrants.
The fact that $\overline{\Cmp_0} \setminus \{ 0 \}$ is not path-connected essentially simplifies
the description of $ \supp \dd M_0  $ and the trajectory $\Phi [\I]$
in comparison with Proposition \ref{p billiard}.
In particular, $\supp \dd M_0 = \{ a_j \}_{j=1}^n$ with $n=1$ or $n=2$,
i.e., $\dd M_0$ consists of at most of two point masses.
Moreover,  $ \aone = a_1 < a_2 = \ell $ (in the settings of statements (i)-(ii) of Proposition \ref{p billiard}).

\begin{prop} \label{p oinL}
Let (\ref{e om loc-bd}), (\ref{e 0 in L}) be fulfilled and $\re \om >0$.
Then $ \aone < \ell $ and the following statements hold:
\item[(i)] Either $ \supp \dd M_0 = \{ \aone \} $, or $\supp \dd M_0 = \{ \aone , \ell  \} $.
\item[(ii)] The trajectory $\Phi [\I]$ is the closed line-segment $[1, \Phi (\ell)]$ (recall that $\Phi (\aone) =1$).
\item[(iii)] $\re \Phi (\ell) \ge 0$ and $\im \Phi (\ell) > 0$. Moreover,
$\re \Phi (\ell) = 0$ in the case $\ell \in \supp \dd M_0$.
\end{prop}

\begin{proof}
By (\ref{e aone<l}), $ \aone < \ell $. It follows from (\ref{e arg phi inc}),
that $\Phi (x)$ and $\Phi^2 (x)$ are in $\CC_+$ for $x$ greater than $\aone$ and sufficiently close to $\aone$.
This, (\ref{e conv Trj 2}), (\ref{e 0 in L}), and Lemma \ref{l phi R} (ii)
yields that $\Phi^2 [\I] \subset \overline{\CC_+} \setminus \{0\}$
and $\Phi [\I] \subset \overline{\QQ_I} \setminus \{0\}$.
In particular, $0 \le \re \Phi (\ell)$.
Applying  (\ref{e arg phi inc}) again, we see that
$0< \im \Phi (x) $ for all $ x \in (\aone,\ell] $,  and $ 0 < \re \Phi (x) $ for all $ x \in (\aone,\ell) .$
This and Proposition \ref{p z0} (i) prove statements (i) and (iii).
Statement (ii) follows from statement (i) and Lemma \ref{l phi on int M=0}.
\end{proof}

\subsection{Proofs of Theorems \ref{t AM opt} and \ref{t A1}}
\label{ss proofs opt}

When $\re \om \neq 0 $, Theorem \ref{t AM opt} follows immediately from Propositions
\ref{p billiard} and \ref{p oinL}.
Consider the special case $\re \om = 0$.

\begin{lem}[\cite{Ka13_KN}, see also \cite{CZ95} and Section 4.1 in \cite{KN89}] \label{l ab in Kr a=0}
Let $\dd M \in \Mes_+$ and  $ (- \ii) \beta \in \Si (\dd M)$. Then
$\beta \ge \| \dd M \|^{-1} $. The equality holds if and only if
$\dd M $ consists of a single point mass placed at $\ell$ .
\end{lem}

This lemma easily yields the following proposition.

\begin{prop} \label{p al=0 min dec}
$\om = -\ii m^{-1} $ is the quasi-eigenvalue of minimal decay for the zero frequency ($\alpha =0$) over $\Am$.
The corresponding string of minimal decay is unique and equals $\dd M_0 = m \de (x- \ell )$.
\end{prop}

This completes the proof of Theorem \ref{t AM opt}.

Now we consider optimization over $\A= \A_1$ and prove Theorem \ref{t A1}.

First, we show that
\begin{equation} \label{e no min a=0 A1}
\text{there is no a string of minimal decay rate for the zero frequency (over $\A_1$).}
\end{equation}

\begin{prop} \label{p const struct}
Let $x_0 \in [0, \ell)$,  $B (x) = 0 $ for $x \in [0, x_0]$ and $B(x) = c $  for $x \in (x_0, \ell]$, where $c>1$ is a constant.
Then
the sequence
$
\displaystyle \om_k = -\ii \frac{1}{2(\ell-x_0)\sqrt{c}} \ln  \frac{\sqrt{c}+1}{\sqrt{c}-1}  + \frac{k\pi}{(\ell-x_0)\sqrt{c}}
$, \ $ k \in \ZZ$,
forms the set of quasi-eigenvalues of the string $B(x) \dd x$.
\end{prop}

The proposition can be obtained by direct computation, see e.g. \cite{CZ95}.

Taking $B(x)$ introduced above with $c$ going to $+\infty$ and $x_0$ such that $(x_0 - \ell) c = m$, one can see
that $\re \om_0 = 0$ (and so $\alpha = 0$ is an admissible frequency) and that $\om_0$ tends to $(-\ii) m^{-1}$. On the other side,
Lemma \ref{l ab in Kr a=0} implies that $(-\ii) m^{-1}$ is not a quasi-eigenvalue for any string in $\A_1$.
This proves (\ref{e no min a=0 A1}).
Note that the above arguments provide an optimizing sequence (see \cite{Ka13_KN}) for $\alpha = 0$.

In the case when the frequency $\alpha \neq 0$, Theorem \ref{t A1} follows immediately from the following result.

\begin{prop} \label{p LocInt}
Assume that $\re \om \neq 0$ and $\om$ is a quasi-eigenvalue of $\dd M_0 \in \A_1$.
Then $\om$ is a strongly $\dd M_0$-local interior point of $\Sigma (\A_1)$. In particular, the set $\Si [\A_1] \setminus \ii \RR$ is open.
\end{prop}

\begin{proof}
Assume that $\om$ is strongly $\dd M_0$-local boundary point of $\Sigma (\A_1)$.
Then $\aone = \aone (\dd M_0)$ is an isolated point of $\supp \dd M_0$.
This fact follows from the proofs of Lemma \ref{l aone 0ninL new} and Proposition \ref{p oinL}
(due to Proposition \ref{p z0}, arguments of these proofs work without changes for the case of the admissible family $\A_1$).
Hence the measure $\dd M_0$ is not absolutely continuous and so $\dd M_0 \not \in \A_1$, a contradiction.
Thus, each $\om \in \Si [\A_1] \setminus \ii \RR$ is a local interior point
(for every string in $\A_1$ that produce the quasi-eigenvalue $\om$). This yields $\om \in \Intr (\Si [\A_1] \setminus \ii \RR)$.
\end{proof}

\section{Weakly* local boundary and the  "reflection law"}
\label{s weak loc-b}

Propositions \ref{p billiard} and \ref{p oinL}  leave too much freedom in the choice of masses $m_j$ and the position $a_1$ of the first point mass.
In this section, we study $\dd M_0 \in \Am$ and $\om \in \QQ_{IV} $ (i.e, $\re \om >0$, $\im \om < 0$) under the stronger assumption that
\begin{eqnarray}
\om \text{ is a weakly* $\dd M_0$-local boundary point of } \Si [\Am] \label{e weak loc-bd}.
\end{eqnarray}
We show that $m_j$, $a_j$, and $\om$ are connected by additional relations.
Then, the results on weakly* local boundary points can be immediately applied to strings of minimal decay
due to the implication
(\ref{e impl loc bd md}) $\Rightarrow$ (\ref{e impl loc bd weak}).
Note that, from the point of view of the minimal decay, only the case $\re \om > 0$
is interesting since the case $\re \om =0$ is described by Proposition \ref{p al=0 min dec}
(see also Remark \ref{r sym}).

\subsection{The "reflection" law and a position of the first point mass}
\label{ss res weak*}

Assumption (\ref{e weak loc-bd}), the implication (\ref{e impl loc bd weak}) $\Rightarrow$ (\ref{e impl loc bd strong}),
and Propositions \ref{p billiard}, \ref{p oinL} imply that
\begin{equation} \label{e dM=sum}
\dd M_0 = \sum_{j=1}^n m_j \delta (x-a_j) \dd x
\end{equation}
with $n \in \NN$,  positive masses $m_j$, and distinct increasingly ordered positions $a_j$,
and also connect $a_j$ with supporting lines $L$ of Propositions \ref{p billiard} and \ref{p oinL}.
This constitute the beginning of the following theorem. The essentially new part of the theorem consists of statements (ii)-(iii).

In the sequel, $n$ is always the number of points in $\supp \dd M_0$.
As before, $  \Phi (x) := \vphi (x,\om; \dd M_0)  $ and $\aone = \aone (\dd M_0) = \min \supp \dd M_0 = a_1$.

\begin{thm} \label{t AM opt det}
Let $\alpha = \re \om >0$. Assume that $\om$ and $\dd M_0$ satisfy (\ref{e weak loc-bd})
(in particular, the latter holds if $\om$ and $\dd M_0$ are of minimal decay for $\alpha $).
Then $\dd M_0$ has the form (\ref{e dM=sum})
with $n \in \NN$,  $m_j>0$, and distinct increasingly ordered $a_j$.
Moreover, $a_j$, $m_j$, and the mode $\Phi$ are connected by the following statements:

\item[(i)] There exists a line $L \subset \CC$ passing through $z_0 = m^{-1} \int_{0-}^{\ell+}  \Phi^2 (x)  \dd M_0 (x) $
and there exists a unit normal $p $ to $L$  such that:

\subitem (i.a)  The trajectory $\Phi [\I]$ is a piecewise linear path lying in
\[
\overline{\Cmp_0} := \{ \zeta \in \CC  \ : \ \langle \zeta^2 -1 , p \rangle_\CC \le 0  \}
\] and consisting of the
closed line-segments $[\Phi (a_j) , \Phi(a_{j+1})]$, $j=1, \dots, n$.
Here $a_{n+1} :=\ell$ (note that $a_{n+1} \not \in \supp \dd M_0$ whenever $a_n < \ell$).
\\[0.2mm]
\subitem (i.b)
In the case $a_n < \ell $, the formula
$ \
 \{ a_j \}_{j=1}^n    =  \{ x \in [\aone , \ell) \ : \ \Phi^2 (x) \in L \}  \
$ holds.
In the case $a_n = \ell$, the formula
$
\ \{ a_j \}_{j=1}^{n}  =  \{ x \in [\aone , \ell] \ : \ \Phi^2 (x) \in L \}  \
$
takes place.
Note that, in both the cases, $1 = \Phi (a_1) \in L$ and so $L=1+ \ii p \RR$.\\

\item[(ii)] When $n \ge 2$, the line $L$ and the unit normal $p$ are uniquely determined by assertions (i.a)-(i.b)
and the following statements hold:

\subitem (ii.a) $a_1 = 0$

\subitem(ii.b) Let the real numbers  $ s_j $ $ (j=1, \dots, n)$  be defined by the equalities $\Phi^2 (a_j) = 1+ \ii p s_j$.
Then $ 0 = s_1 >  s_2 > \dots > s_n $.

\subitem (ii.c) When $j \ge 2 $ and $a_j < \ell$,
\begin{equation} \label{e mj a neq l}
m_j = \frac{ \langle \ \pa_x^- \Phi^2 (a_j) \ , \ p \ \rangle_\CC}{\langle \ \om^2  \Phi^2 (a_j)  \  , \ p  \ \rangle_\CC} =
- \frac{ \langle \ \pa_x^+ \Phi^2 (a_j) \ , \ p \ \rangle_\CC}{\langle \ \om^2 \Phi ^2 (a_j)  \ , \ p  \ \rangle_\CC} .
\end{equation}

\subitem (ii.d) When $a_n = \ell$,
\begin{equation} \label{e mn a=l}
m_n = \frac{ \pa_x^- \Phi (a_n) }{ \om^2  \Phi (a_n)  } - \frac{\ii}{\om} .
\end{equation}

\item (iii) If $\| \dd M_0 \|<m$, then $L = \RR$, $n \le 2$, and $a_1 = 0$.
\end{thm}

\begin{proof}[Proof of statement (i) of Theorem \ref{t AM opt det}.]
Only (i.b) needs an additional argument.  It was proved that  $\Phi^2 [ \supp \dd M_0 ] \subset L$. Assertion (i.b) means that $\wt x \in [\aone,\ell)$ and $\Phi^2 (\wt x) \in L$ imply $\wt x \in \supp \dd M_0$.
Let us prove this by \emph{reductio ad absurdum}. Assume $\wt x \not \in \supp \dd M_0$. Then $\wt x \in (a_j, a_{j+1})$ for certain $1 \le j \le n$. By (ia), $\Phi (\wt x) \in \left( \Phi (a_j), \Phi (a_{j+1}) \right)$. On the other side, $\Phi (\wt x) \in \Hyp$. So the line passing through $\Phi (a_j)$ and $ \Phi (a_{j+1})$ is a tangent line to $\Hyp$ at $\Phi (\wt x)$. On the other side, it intersects $\Hyp$ at least twice (at $\Phi (a_j)$ and $\Phi (\wt x)$), a contradiction.
\end{proof}

\begin{rem} Actually, it is proved that (i) is valid if  (\ref{e weak loc-bd}) is replaced by  the weaker assumption (\ref{e om loc-bd}).
\end{rem}

Statements (ii) and (iii) will be proved in Subsection \ref{ss proof weak* lb}.

It follows from (\ref{e pa+ = pa- -om}) that
\begin{equation} \label{e pa+ = pa- -om aj}
 \pa_x^+ \Phi (a_j) - \pa_x^- \Phi (a_j) = - \om^2 \Phi (a_j) m_j .
\end{equation}
In the case when $j \ge 2$ and $a_j < \ell$,  the proof of (ii.c) shows that
\begin{equation} \label{e pa pm proj  p}
 \langle \pa_x^+ \Phi^2 (a_j) \ , p \ \rangle_\CC
 = - \langle \pa_x^- \Phi^2 (a_j) \ , p \ \rangle_\CC  .
\end{equation}
The combination of the last equality and (\ref{e pa+ = pa- -om aj}) can be interpreted
as \emph{a nonstandard reflection law} for the hyperbolic billiard of Section  \ref{sss HypBill}.

\subsection{Weakly*-continuous coordinates on N-dimensional faces}
\label{ss weak coord}

Let $N \in \NN$. 
With two $N$-tuples $b=(b_1,b_2, \dots, b_N) \in \RR^N$ and $\mu = ( \mu_1,\mu_2, \dots, \mu_N) \in \RR^N$ we associate
the measure
$
\displaystyle \ \dd \Delta_{b,\mu} := \sum_{j=1}^N \mu_j \delta (x-b_j) \dd x. \
$
Then (\ref{e int ep phi}) implies
\begin{equation} \label{e phi disc}
\vphi (x, z; \dd \Delta_{b,\mu} ) =  1- z^2 \sum_{b_j <  x} (x-b_j ) \ \vphi (b_j, z; \dd \Delta_{b,\mu} ) \ \mu_j .
\end{equation}

We say that $b$ is \emph{a tuple of ordered positions} if
\begin{eqnarray}
 0 \le b_1 \le b_2 \le \dots \le b_N \le \ell . \label{e b<b}
\end{eqnarray}
\emph{The set of tuples of ordered positions} is denoted by $\PP_N$. It is convex.


The following lemma concerns analyticity of $F (z;  \dd \Delta_{b,\mu} )$ in
$b_j$ and $\mu_j$.

\begin{lem} \label{l Fdisc}
For each $N \in \NN$, there exists polynomials $F_\disc (z; b, \mu )$, $\vphi_1 (z; b, \mu ) $, \dots, $\vphi_N (z; b, \mu ) $,
in the $2N+1$ variables
$z$, $b_1$, \dots, $ b_N $, $ \mu_1 $, \dots $\mu_N$, such that for any $b \in \PP_N$ the equalities
\[
F (z;  \dd \Delta_{b,\mu} ) = F_\disc (z; b, \mu ) \ \text{ and  } \
\vphi (b_j , z ;  \dd \Delta_{b,\mu} ) = \vphi_j (z; b, \mu ) , \ \ \ j=1,\dots, N
\]
hold (note that the polynomials depend also on $N$).
\end{lem}

\begin{proof}
Assume (\ref{e b<b}).
Then (\ref{e phi disc}) takes the form
\begin{equation*} 
\vphi (b_j , z ;  \dd \Delta_{b,\mu} )  =  1- z^2 \sum_{k=1}^{j-1} \ (x-b_j ) \  \vphi_k (z ; b , \mu ) \ \mu_j  .
\end{equation*}
In particular, $\vphi (b_1 , z ;  \dd \Delta_{b,\mu} )  = 1 =: \vphi_1 (z; b, \mu )$ and one can inductively show that
$\vphi (b_k , z ;  \dd \Delta_{b,\mu} ) $ are polynomials in $z$, $b_j$, and $\mu_j $ for all $k =1$, \dots, $N$.
We denote them by $\vphi_k (z; b, \mu )$.
Next, (\ref{e phi disc}) and (\ref{e b<b}) yield that $\vphi (\ell, z; \dd \Delta_{b,\mu} ) $ is a polynomial.
It follows from (\ref{e phi disc}) (or directly from (\ref{e pa phi 1})) that
\begin{equation*} 
\pa_x^+ \vphi (x, z; \dd \Delta_{b,\mu} ) =
- z^2 \sum_{b_j \le  x} (x-b_j ) \ \vphi (b_j, z; \dd \Delta_{b,\mu} ) \ \mu_j   .
\end{equation*}
Thus, $z^{-1} \pa_x^+ \vphi (\ell, z; \dd \Delta_{b,\mu} ) $ and, in turn, $F (z;  \dd \Delta_{b,\mu} )$ are also polynomials.
\end{proof}


Assume that $b^0 = (b^0_j )_{j=1}^N \in \PP_N$ is fixed.
A tuple $v= (v_j) \in \RR^N$ is \emph{a tuple of $b^0$-order preserving velocities} if there exists $\ep>0$
such that the $t$-dependent tuple
\begin{equation} \label{e b(t)}
b(t) = \left( b_j (t) \right) := (b_j^0 + t v_j)
\end{equation}
belongs to $\PP_N$ for all $t \in [0,\ep]$. The set of tuples of $b^0$-order preserving velocities
is denoted by $\V = \V (b^0)$. Obviously,
\[
\V (b^0) = \cone (\PP_N - b^0).
\]

Fixing $\mu$ and taking $z= \om$,
consider the function $F (\om; \dd \De_{b(t),\mu})$ with $b(t)$ as in (\ref{e b(t)}). This function depends only on $t$.
Its right derivative taken at $t=0$ is a functional of the velocities tuple $v$. It will be denoted by
\[
\DF_{b^0,\mu} (v) := \left[ \pa_t^+ F (\om, \dd \De_{b(t),\mu}) \right]_{t=0} .
\]
According to Lemma \ref{l Fdisc} this derivative exists (at least) for $v \in \V (b^0)$.
The set
\[
\DF [\V  ] = \DF_{b^0,\mu} [\V (b^0) ] := \{ \DF_{b^0,\mu} (v)  \ : \ v \in \V (b^0) \}
\]
of such derivatives is a convex cone in $\CC$.

One can see from (\ref{e dM=sum}) that for every $N_0 \ge n$
there exists a tuple of ordered positions $b^0 \in \PP_{N_0}$ and a tuple
of nonnegative numbers $\mu^0 \in \Rnneg^{N_0}$
such that $\dd M_0 = \dd \Delta_{b^0,\mu^0} $.

\begin{prop} \label{p w* boundary}
Assume $\re \om >0$ and  (\ref{e weak loc-bd}).
Let the set $S_0$ and the complex numbers $C_0$, $z_0$ be defined as in (\ref{e S0}), (\ref{e C0}), (\ref{e z0}).
Let $b^0 \in \PP_{N_0}$ and $\mu^0 \in \Rnneg^{N_0}$ be
such that $\dd M_0 = \dd \Delta_{b^0,\mu^0} $.
Then
\[
\text{$0$ is a boundary point
of the convex cone } \quad
 \Cone_0 := \cone \left( C_0 \DF [\V ] \cup [S_0 -z_0] \right)
\]
generated by the convex sets
$C_0 \DF [\V ] = \left\{ C_0 \DF_{b^0,\mu^0} (v)  \ : \ v \in \V (b^0) \right\}$
and  $S_0 -z_0$.
\end{prop}

\subsection{Proof of Proposition \ref{p w* boundary}} 
\label{ss pr p w* boundary}


It is enough to prove that $0 \in \Bd S_*$, where
\begin{equation} \label{e S*}
S_* := \conv \left( \frac{C_0}{m} \DF [\V ] \cup [S_0 -z_0] \right).
\end{equation}
This fact can be easily obtained from the following three lemmas.

For a tuple $b = (b_j) \in \RR^N$,
we denote by $\Phi^2 [b] $ the set consisting of $\Phi^2 (b_1)$, \dots, $\Phi^2 (b_N)$.

\begin{lem} \label{l 0 in Int S2}
Assume  $0 \in \Intr S_* $. Then there exist $N \ge N_0$, $b^1 \in \PP_{N}$
and $\mu^1 \in \Rnneg^{N}$ such that
$\dd M_0 = \dd \Delta_{b^1,\mu^1} $  and  $0 \in \Intr S_2$,
where
\begin{equation} \label{e S2}
S_2 := \conv \left( \ \frac{C_0}{m} \DF_{b^1,\mu^1} [\V (b^1) ] \cup S_1 \ \right)
\text{ and } S_1 := \conv \left( \Phi^2 [b^1] \cup \{0\} \right) - z_0 .
\end{equation}
\end{lem}

\begin{proof}
Since $0 \in \Intr S_* $, it is possible to take $\zeta_1, \zeta_2, \zeta_3 \in S_*$ such that
$0 \in \Intr \, \conv \{ \zeta_1,\zeta_2,\zeta_3 \} $.
By (\ref{e S*}) and (\ref{e conv Trj 2}), $\zeta_j$ can be represented as convex combinations
\begin{equation} \label{e zetaj comb}
\zeta_j = \la_{j}^{[-1]} \frac{C_0}{m} \DF_{b^0,\mu^0} (v^j) - \la_{j}^{[0]} z_0 + \sum_{k=1}^{n_j} \la_{j}^{[k]}
[\Phi^2 (x_{j,k}) - z_0]
\end{equation}
with certain $v^j \in \V (b^0)$, $n_j \ge 0 $, $x_{j,k} \in \I $ and certain $\la_{j}^{[k]} \ge 0 $ satisfying
$\sum_{k=-1}^{n_j} \la_{j}^{[k]} = 1 $.
Let us construct a new tuple $b^1 \in \PP_{N+ n_1+n_2+n_3 }$ adjoining
$( x_{j,k} )^{k =1,\dots, n_k}_{j=1,2,3}$ to $b^0=( b^0_i )_{i=1}^N$ and then sorting the obtained tuple
in the increasing order. The associated tuple $\mu^1 \in \Rnneg^{N+ n_1+n_2+n_3 }$ is constructed
by insertion of $0$ into the tuple $\mu^0 =( \mu^0_i )_{i=1}^N$ at the places where $x_{j,k}$ were inserted into
$b^1$. So
\[
\dd \Delta_{b^1,\mu^1} = \dd \Delta_{b^0,\mu^0} = \dd M_0.
\]
Since any movements of zero point masses $\mu^1_i = 0$ inserted at $b^1_i = x_{j,k}$  do not influence the measure
$\dd \Delta_{b^1,\mu^1}$, it is easy to see that $\DF_{b^1,\mu^1} [\V (b^1) ] = \DF_{b^0,\mu^0} [\V (b^0) ] $.
This, (\ref{e zetaj comb}), and $0 \in \Intr \conv  \{ \zeta_1,\zeta_2,\zeta_3 \} $ yield that
$0$ is an interior point of $S_2$.
\end{proof}

Denote by $\A^\disc_N $ the set of $(b,\mu) \in \PP_{N} \times \Rnneg^{N}$ such that
$\dd \Delta_{b,\mu} \in \Am$.
Clearly, $\A^\disc_N $ is convex.
Let  $F_\disc $ be the polynomial defined in Lemma \ref{l Fdisc}.

\begin{lem} \label{l strong to weak}
Let $(b,\mu) \in \A^\disc_N$ be such that $\dd M_0 = \dd \Delta_{b,\mu}$.
Assume that $\om$ is a (strongly) $(b,\mu)$-local interior point of $\Si_{F_\disc} [\A^\disc_N]$.
Then $\om$ is a weakly* $\dd M_0$-local interior point of $\Si [\Am]$.
\end{lem}

\begin{proof} Let $W$ be a neighborhood of $\dd M_0$ in the weak* topology $\Tpl_{w*}$.
Take any $W_1 \subset W$ from the standard base of $\Tpl_{w*}$-neighborhoods of $\dd M_0$.
The latter means that there exist $\ep >0$
and a finite family of functions $y_j \in C [0,\ell]$, $j=1, \dots, k$, such that
$W_1 = \{ \dd M \ : \ | \langle y_j , \dd M - \dd M_0 \rangle  | < \ep \ , \ j=1,\dots, k \}$.
Clearly, there exists a neighborhood $W_E$ of $(b,\mu)$ (in
the Euclidean topology of $\RR^{2N}$) with the property
\[
\{ \dd \Delta_{b,\mu} \ : \  (b^1,\mu^1) \in W_E \} \subset W_1 .
\]
This and Lemma \ref{l Fdisc} yield $\Si_{F_\disc} [\A^\disc_N \cap W_E] \subset \Si [\Am \cap W]$.
By the assumption of the lemma, $\om \in \Intr \Si_{F_\disc} [\A^\disc_N \cap W_E]$.
Thus, $\om$ is an interior point of $\Si [\Am \cap W]$.
\end{proof}

\begin{lem} \label{l 0 in Int S2 -> loc int disc}
Assume that $(b^1,\mu^1) \in \A^\disc_N $, $\dd M_0 = \dd \Delta_{b^1,\mu^1} $, and $0 \in \Intr S_2$,
where $S_2$ is defined by (\ref{e S2}).
Then $\om$ is a $(b^1,\mu^1)$-local interior point of $\Si_{F_\disc} [\A^\disc_N]$.
\end{lem}

\begin{proof}
By Theorem \ref{t ab loc bd} and Proposition \ref{p iii'iii''},
it is enough to show that $0 \in \Intr S_2$ implies
that assumption (iii'') of Proposition \ref{p iii'iii''} holds with $G= F_\disc$ and
$S= \A^\disc_N$.

To this end, we extend the polynomial $F_\disc $ to complex $b_j$ and $\mu_j$.
Then $F_\disc $ is an analytic functional of $ (z; b , \mu) \in \CC \times \CC^N \times \CC^N$.
Denote by $\frac{\pa F_\disc }{\pa b}  (v)$ and $\frac{\pa F_\disc }{\pa \mu}  ( \wt \mu)$
the directional derivatives in $b$ and $\mu$, resp., at the point $(\om;b^1,\mu^1)$.
We consider these derivatives as linear functionals of $v \in \CC^N$ and
$\wt \mu \in \CC^N$, respectively.

Obviously,
\[
\frac{\pa F_\disc }{\pa \mu}  (\wt \mu) = \frac{\pa F (\om; \De_{b^1,\mu^1})}{\pa M} \left(\De_{b^1,\wt \mu} \right)
= C_0^{-1} \sum_{j=1}^{N} \Phi^2 (b^1_j) \wt \mu_j.
\]
Hence the set $S_1$ defined in (\ref{e S2}) can be expressed as
$
S_1 =  \left\{ \frac{C_0}{m} \, \frac{\pa F_\disc }{\pa \mu}  (\wt \mu)  \ : \
(b^1, \mu^1 + \wt \mu )
\in \A^\disc_N  \right\} .
$
It follows from $\V (b^0) = \cone (\PP_N - b^0)$
that
$ \cone \left( \frac{\pa F_\disc }{\pa b}  [\PP_N - b^0 ]  \right) = \DF [\V (b^0)] .$

Summarizing, we see that $0 \in \Intr S_2$ implies that
$0$ is an interior point of the convex cone $S_3$ generated by the union of the sets
$\frac{C_0}{m} \frac{\pa F_\disc }{\pa b}  [\PP_N - b^0 ] $ and
$\frac{C_0}{m} \left\{ \frac{\pa F_\disc }{\pa \mu}  (\wt \mu)  \ : \ (b^1, \mu^1 + \wt \mu )
\in \A^\disc_N  \right\} $ (actually, this means that $S_3 = \CC$).
So condition (iii'') of Proposition \ref{p iii'iii''} is fulfilled.
Thus, Theorem \ref{t ab loc bd} yields that $\om$ is a $(b^1,\mu^1)$-local interior point of $\Si_{F_\disc} [\A^\disc_N]$.
\end{proof}

\subsection{Movements of point masses}


\begin{prop} \label{p m one move}
Assume that  $b^0 \in \PP_N$, $v \in \V (b^0)$,  and $1\le k \le N$
are such that
\begin{eqnarray}
& v_j = 0 \ \text{ and } \ b_j^0 \neq b_k^0 \ \ \ \text{ for all } \ \ \ j \neq k  \label{e vj=0 k}.
\end{eqnarray}
Assume that $\mu \in \RR^N$ and $F (\om; \dd \De_{b^0,\mu} ) = 0$.
Let $ b(t) = b^0  + t v$ and $\vphi (x) = \vphi (x, \om; \dd \De_{b^0,\mu})$ .
Then the right $t$-derivative
$\DF := \left[ \pa_t^+ F (\om, \dd \De_{b(t),\mu}) \right]_{t=0} $
is given by
\begin{equation} \label{e Do F k}
 \DF = - \frac{\ii \om  }{ \vphi (\ell)}
 \vphi (b_k^0) \left[\pa_x^- \vphi (b_k^0) + \pa_x^+ \vphi (b_k^0)\right]  v_k  \mu_k   .
\end{equation}
\end{prop}

\begin{proof}
Let $\Phi (x,t) := \vphi (x , \om ; \dd \Delta_{b(t),\mu}) $.
Then
\begin{eqnarray} \label{e Ph xt k}
\Phi (x,t) & = & 1- \om^2 \sum_{b_j (t) <  x} [x-b_j (t)] \ \Phi (b_j (t),t) \ \mu_j ,
\\
\pa_x^+ \Phi (x,t) & = & -\om^2 \sum_{b_j (t) \le x} \ \Phi (b_j (t),t) \ \mu_j . \label{e pax Ph xt k}
\end{eqnarray}
Put
$ \qquad
\Phi_j (t) := \Phi (b_j (t) \, , \, t) , \ \ \
\Phi_j^+ (t) := -\om^2 \sum_{i=1}^{j} \Phi_i (t) \mu_i .
$

For $t\in (0,t_0)$ with $t_0>0$ small enough, we have $b(t) \in \PP_N$ and
\begin{equation}
 \Phi_j (t) =  1- \om^2 \sum\limits_{i=1}^{j-1} [b_j(t) -b_i (t)] \Phi_i (t) \mu_i . \label{e Phj eq k}
\end{equation}

Differentiating (\ref{e Phj eq k}) and (\ref{e Ph xt k}) in $t$ with the use of
(\ref{e vj=0 k}), one gets
\begin{eqnarray} 
 \pa_t^+ \Phi_j (0)  & = & 0   \ \ \text{ if $j<k$ } , \notag
 \\
 \pa_t^+ \Phi_{k} (0) & = & - v_k \om^2 \sum_{i = 1}^{k-1}  \Phi_i (0) \mu_i ,
\label{e Dt Ph k} \\
\pa_t^+ \Phi (x,0) & = & - \om^2 \sum\limits_{\substack{i \geq k \\ b_i^0 < x} }
(x- b_i^0 ) \  \pa_t^+ \Phi_i (0) \ \mu_i +
\om^2 v_k \vphi (b_k^0 )  \mu_k \ \  \ \text{ if $x > b_k^0$}. \label{e D0 Ph k}
\end{eqnarray}
The last equality shows that, for $x \in [b_k^0,+\infty)$,
the function $\theta (x) := \lim\limits_{\wt x \to x+0 } \pa_t^+ \Phi (\wt x,0)$
is  a unique solution to the integral equation
\begin{eqnarray}
& y(x) = c_0 + c_1 (x-b_k^0) - \om^2 \int\limits_{b_k^0 +}^x (x-s) y(s) \dd \De_{b^0,\mu} \
\ \text{ on the interval } x \in [b_k^0,+\infty) \label{e th eq k}\\
& \text{with } \ \ c_0 = \om^2 \vphi (b_k^0 ) v_k  \mu_k   \ \ \text{ and  } \ \
c_1 =  - \om^2 \pa_t^+ \Phi_k (0) \mu_k . \notag
\end{eqnarray}
Note that $\theta (x) = \pa_t^+ \Phi (x,0)$ when $x>b_k^0$,
and $\theta (b_k^0) = \pa_t^+ \Phi (b_k^0,0)$ when $v_k \le 0$.
In the case $b^0_k = \ell$, the inequality  $v_{k} \le 0$ holds due to $v \in \V (b^0)$.

Since the functions $\vphi (x) $ and $\psi (x) := \psi (x , \om ; \dd \Delta_{b^0,\mu}) $
are also solutions of (\ref{e th eq k}) with the "initial data" $(c_0,c_1)$ equal to $\left( \vphi (b_k^0), \pa_x^+ \vphi (b_k^0) \right)$ and
$\left( \psi (b_k^0), \pa_x^+ \psi (b_k^0) \right)$, resp.,
the theory of Volterra integral equations (see e.g. \cite{A64}) states that
\begin{equation} \label{e th= k}
\theta (x) = c_2 \vphi(x) + c_3 \psi (x),  \ \  \text{ where $c_{2,3} $ are certain constants} .
\end{equation}

By definition, $\DF = \pa_t^+ \Phi (\ell,0) +  \frac{i}{\om} \, \pa_t^+ \pa_x^+ \Phi (\ell,0)$.
Taking $\pa_t^+$ of (\ref{e pax Ph xt k}) and $\pa_x^+$ of (\ref{e D0 Ph k}) and (\ref{e th eq k}),
we see that
$
\pa_x^+ \theta (\ell) = \pa_x^+ \pa_t^+ \Phi (\ell,0) = \pa_t^+ \pa_x^+ \Phi (\ell,0).
$
So
$
\DF = \theta (\ell) +  \frac{i}{\om} \, \pa_x^+ \theta (\ell) .
$
Using (\ref{e th= k}) and the equalities $\vphi (\ell) +  \frac{i}{\om} \, \pa_x^+ \vphi (\ell) = 0 $ and (\ref{e i / phi l}),
we get
\begin{equation} \label{e DoF psi k}
\DF = c_3 \left[ \psi (\ell) +  \frac{i}{\om} \, \pa_x^+ \psi (\ell) \right] = \frac{\ii c_3}{\om \vphi (\ell)} .
\end{equation}

The constants $c_2$ and $c_3$ can be expressed via $c_0$ and $c_1$ using (\ref{e th= k}) and
(\ref{e Wrons}) for $x=b_k^0$.
For $c_3$, one gets
\begin{gather*}
c_3 = \left|
\begin{array}{cc}
\vphi (b_k^0) & c_0 \\
\pa_x^+ \vphi (b_k^0) & c_1
\end{array}
\right|
= - \om^2  \vphi (b_k^0)
\left|
\begin{array}{cc}
1  &  - v_k  \mu_k \\
\pa_x^+ \vphi (b_k^0) &    \pa_t^+ \Phi_k (0) \mu_k
\end{array}
\right| =
- \om^2  \vphi (b_k^0) I_1 ,
\end{gather*}
where
$
I_1 =
\pa_t^+ \Phi_k (0) \mu_k + \pa_x^+ \vphi (b_k^0)  v_k  \mu_k   .
$
Using (\ref{e Dt Ph k}) and the equality
$\pa_x^- \vphi (b_k^0) = -\om^2  \sum_{i = 1}^{k-1}
\Phi_i (0) \mu_i$, one gets
$
I_1 =
\left[\pa_x^- \vphi (b_k^0) + \pa_x^+ \vphi (b_k^0)\right]  v_k  \mu_k  .
$
Combining this with (\ref{e DoF psi k}), we complete the proof.
\end{proof}

\subsection{Proof of statements (ii)-(iii) of Theorem \ref{t AM opt det}}
\label{ss proof weak* lb}

\textbf{First, we consider the case $n=2$ and prove statement (ii)}.
Propositions \ref{p billiard} (v) and \ref{p oinL} imply that the set $\{ \Phi (a_j) \}_1^n$ consists
of $n$ distinct points. This yields the uniqueness of the line $L$ satisfying (i.a)-(i.b) and the fact that $L$ is the unique supporting line to the set $ S_0 $ at $ z_0 = m^{-1} \sum \Phi^2 (a_j) m_j$.

\emph{Statement (ii.b)} follows from Propositions \ref{p billiard} (v) and \ref{p oinL}. It follows from  (ii.b) that \linebreak
$ \{ \Phi^2 (a_j) \}_{j=1}^n   \subset  [\Phi^2 (a_1) , \Phi^2 (a_n) ] $. In turn, this yields
\begin{eqnarray}
& z_0  \in  \bigl( \Phi^2 (a_1) , \Phi^2 (a_n) \bigr). \label{e z0 in segment}
\end{eqnarray}
Indeed, when $\| \dd M_0 \| = m$, one can see that $z_0$ is a convex combination of $ \Phi^2 (a_j) $
with nonzero coefficients $m_j/m$. So (\ref{e z0 in segment}) follows directly from $ \{ \Phi^2 (a_j) \}_{j=1}^n   \subset [\Phi^2 (a_1) , \Phi^2 (a_n) ] $.
Consider the case $\| \dd M_0 \| < m$. By Proposition \ref{p oinL}, one has $L = \RR$ and
\begin{equation} \label{e L=R n=2}
 \Phi^2 (a_1) = 1 >0, \qquad \qquad \Phi^2 (a_2) = \Phi^2 (\ell) < 0.
\end{equation}
Since $m_1/m$ and $m_2/m$ are less than $1$, we again get (\ref{e z0 in segment}).

It is easy to get from (\ref{e z0 in segment}) and (i.b) that
\begin{equation} \label{e H_0 in cone}
\overline{H_0} - z_0 = \cone (S_0 -z_0) ,
\end{equation}
where $H_0$ is the closed half-space  defined in Sections \ref{sss HypBill} and \ref{sss deg case}.
Indeed, (i.b) implies that the points of
$\Phi^2 \left[ \I \setminus \{ a_j \}_{j=1}^n \right] $ do not belong to $\Bd H_0 = L$.
By (\ref{e Trj in}), these points are in $ H_0 $. So the inclusion (\ref{e S0 sub H0}) yields
$S_0 \cap H_0 \neq \emptyset$. This, (\ref{e z0 in segment}), and
$
[\Phi^2 (a_1) , \Phi^2 (a_n) ] \subset S_0 \subset \overline{H_0}
$
imply (\ref{e H_0 in cone}).

Writing (\ref{e dM=sum}) with the use of the tuples of positions $a:=(a_j)_1^n \in \PP_n$ and masses $\mu = (m_j)_1^n \in \RR_+^n$, we see that $\dd M_0 = \dd \De_{a,\mu}$.
Combining (\ref{e H_0 in cone})  with Proposition \ref{p w* boundary} (applied to $\mu^0 = \mu $, $b^0 = a$ and $\V = \V (a)$), we see that
$0$ is a boundary point of
$\Cone_1 := \cone \left( C_0 \DF [\V ] \cup [\overline{H_0} - z_0] \right).$
So $\Cone_1 \neq \CC $ and
\begin{equation} \label{e DF V in H0}
C_0 \DF [\V ] \subset \overline{H_0} - z_0 .
\end{equation}

\emph{Let us prove (ii.a)}.
Assume that $a_1 > 0$.
Then $(v_j)_1^n$ is a tuple of $a$-order preserving velocities whenever
$v_j = 0$ for all $j \ge 2$ and $v_1 < 0$ (i.e., we can move $a_1$ slightly back keeping the tuple of positions ordered). Put $v_1 = -1$.
It follows from $\Phi (a_1)=1$,  $\pa_x^- \Phi (a_1)=0$,  $\pa_x^+ \Phi (a_1)= - \om^2 m_1 $, and
Proposition \ref{p m one move}
that
\begin{equation*} \label{e D move a1}
C_0 \DF_{a,\mu} (v) =  - \pa_x^+ \Phi (a_1)  m_1 = \om^2 m_1^2  .
\end{equation*}

From $n \ge 2$, Proposition \ref{p billiard} (v),  and Proposition \ref{p oinL}, we see that $0 = s_1 > \frac{\left< \om^2 , p \right>_\CC}{\im \om^2} $.
Since  $\im \om^2 < 0$, this gives $\left< \om^2 , p \right>_\CC >0$.
 The latter is equivalent to $\om^2 \not \in \overline{H_0} - z_0$. So $C_0 \DF_{a,\mu} (v) \not \in \overline{H_0} - z_0$.
This contradicts (\ref{e DF V in H0}).

\emph{Let us prove (ii.c)}.
Define the tuple $v = (v_j)_1^n$ by $v_k = 0$ for $k \neq j $ and $v_j= 1 $.
Since $a_{j-1} < a_j < \ell$, we see that $\pm v$ are tuples of $a$-order preserving velocities.
Proposition \ref{p m one move} and (\ref{e im vhi vhi'}) yield
\begin{equation*} 
C_0 \DF_{a,\mu} (v) =  \Phi (a_j) [\pa_x^- \Phi (a_j) + \pa_x^+ \Phi (a_j)] m_j \neq 0  .
\end{equation*}
Since $\DF_{a,\mu} (-v) = - \DF_{a,\mu} (v)$, the inclusion (\ref{e DF V in H0})
implies that $\pm C_0 \DF_{a,\mu} (v)$ are parallel to $L$. This can be written as
(\ref{e pa pm proj  p}).
Combining (\ref{e pa pm proj  p}) and (\ref{e pa+ = pa- -om aj}),
one gets
$
 \langle \pa_x^- \Phi^2 (a_j) \, , \, p \ \rangle_\CC
 =  m_j \langle \om^2 \Phi^2 (a_j) \, , \, p  \ \rangle_\CC ,
$
and then (\ref{e mj a neq l}).

\emph{Statement (ii.d)} follows from (\ref{e pa+ = pa- -om aj}) and (\ref{e BCl}).

\textbf{Finally, consider the case $\| \dd M_0 \| < m$ and prove statement (iii).}
By Propositions \ref{p z0} and \ref{p oinL}, $\RR$ is the only supporting line to $S_0$ at $z_0$
and $n \le 2$. Moreover, $z_0 \in (-\infty,1)$. Indeed,
in the case $n=2$, the latter follows from (\ref{e L=R n=2}). In the case $n=1$, from
 $m_1 = \|\dd M_0 \| < m$ and $\Phi (a_1) =1$.
Since every line $L$ satisfying (i) contains $z_0$ and $1$, we see that $L = \RR$.

It is easy to show that $\cone (S_0 -z_0) = \overline{\CC_+} $.
Indeed, in the case $n=1$, this  follows from $[0,1] \subset S_0$ and $z_0 \in (0,1)$. The case $n=2$ has been considered in the proof of (ii), see (\ref{e H_0 in cone}).

Assume $a_1 > 0$. Then, similarly to the proof of (ii.a),
one can show that $\om^2 m_1 \in C_0 \DF [\V]$. This and $\om^2 \in \CC_-$ yield
that $0 \in \Intr \Cone_0$. The latter contradicts Proposition \ref{p w* boundary}.

\section{Calculation of optimizers for small frequencies}
\label{s small fr}

Recall that, when $\dd M_0$ has the form (\ref{e dM=sum}),  $n$ stands for the number of points in $\supp \dd M_0$.

Proposition \ref{p Sclosed} and the implication (\ref{e impl loc bd md}) $\Rightarrow $ (\ref{e impl loc bd weak}) yield the following statement.

\begin{prop} \label{p calc of min from w*bd}
Let $\Bdlocw \Si [\Am]$ be the set of $\om$ that satisfy (\ref{e weak loc-bd}) for certain $\dd M_0 \in \Am$. Then
$\displaystyle \
\beta_{\min} (\alpha; \Am) = \min \{ -\im \om \ : \ \om \in \Bdlocw \Si [\Am] \text{ and } \re \om = \alpha \} \
$
(the minimum exists for each $\alpha \in \RR$).
\end{prop}

Let $\beta_1 : \RR \to \RR_+$ be an even function defined by
the equalities
\begin{eqnarray*}
\text{$\beta_1 (0) := m^{-1}$, \qquad $ \beta_1 (\alpha) := \frac{1}{2\ell} $ \ \ when \ \
$\alpha \neq 0$ and
$\alpha^2 \ge \frac{1}{m \ell } - \frac {1}{4 \ell^2} $, }  \label{e b1 1}\\
\text{and (in the case $m<4\ell$) \ \
$\beta_1 (\alpha) := m^{-1} + \sqrt{m^{-2} - \alpha^2}$
\ when \ $0<\alpha^2 < \frac{1}{m \ell } - \frac {1}{4 \ell^2} $}.
\label{e b1 2}
\end{eqnarray*}

\begin{thm} \label{t small fr}
Let $-(m\ell)^{-1/2} \le \alpha \le (m\ell)^{-1/2}$. Then:
\item[(i)] $ \beta_{\min} (\alpha; \Am) = \beta_1 (\alpha) $
\item[(ii)] There exists a unique string $\dd M^{[\alpha]}$ of minimal decay for the frequency $\alpha$.
\item[(iii)]
$\dd M^{[\alpha]} = m_1 \de (x- a_1 ) \dd x$ with
$m_1 $ and $a_1 $ given by
\begin{eqnarray}
& m_1 =m, \ \ \ a_1 = \ell  & \text{ when }
\alpha = 0;
\notag \\
& m_1 =m, \ \ \ a_1 = \ell - \frac{1}{2 m^{-1} + 2 (m^{-2} - \alpha^2)^{1/2}}  & \text{ when }
0<\alpha^2 < \frac{1}{m \ell } - \frac {1}{4 \ell^2} ;
\notag \\
& m_1 =\frac{1}{(4\ell)^{-1}+ \alpha^2 \ell}, \ \ \ a_1 = 0   & \text{ when }
\alpha \neq 0 \text{ and } \alpha^2 \ge \frac{1}{m \ell } - \frac {1}{4 \ell^2}  .
\label{e m1a1 case3}
\end{eqnarray}
\end{thm}

The proof is given in the Subsection \ref{ss proof small fr}.

\subsection{The case of two supporting lines}

\begin{lem} \label{l two support lines}
Let (\ref{e om loc-bd}) be fulfilled and $\re \om >0$.
Then the following statements are equivalent:
\item[(i)] There exist two distinct supporting lines to the set $ S_0 $ at $ z_0 $.
\item[(ii)] $\dd M_0 = m \delta (x-a_1) \dd x $ and $\arg_0 (\Phi^2 (\ell) -1) \neq \arg_0 \om^2$.
\end{lem}

\begin{proof}
From Proposition \ref{p z0}, we see that there exists at leat one supporting line
to $ S_0 $ at $ z_0 $, that every such a line contains $\{ \Phi (a_j) \}_1^n$, and,
in particular, contains $1=\Phi (a_1) $. So every supporting line has the form $L (p) := 1+ \ii p \RR $, where
$p$ is a univ normal vector to $L (p)$.
Let $\pset $ be the set of $p \in \TT$ such that $S_0 \subset \overline{H_0 (p)}$ and $z_0 \in L (p)$ (see Sections \ref{sss HypBill} and \ref{sss deg case} for
the definition of $H_0 (p)$).
So $\{ L_p \}_{p \in \pset}$ is the set of all supporting lines to $S_0$ at $z_0$.
Clearly, $\pset \neq \emptyset$ and $\pset \subset \{ e^{\ii \xi} : \xi \in [-\pi/2, \pi/2) \}$ (since $0 \in S_0 \subset H_0 (p)$).

\emph{Step 1: statement (i) $\Rightarrow $ n=1 and $m_1 =m$.}
Suppose $n \ge 2$.
Since the points $\Phi^2 (a_1)$ and $\Phi^2 (a_2)$ are distinct and
belong to $L (p)$ for each $p \in \pset$, we see that there exists only one supporting  line to $ S_0 $ at $ z_0 $, a contradiction.
So $n=1$. Suppose $ m_1 < m$. By Proposition \ref{p z0} (iii), $\RR$ is the only supporting line to $ S_0 $ at $ z_0 $, a contradiction.

\emph{Step 2: restrictions on $\pset$ in the case when $n=1$ and $m_1 = m$. }
Suppose $n=1$ and $m_1 = m$.
In this case, $z_0 = 1$. In particular, $z_0 \in L(p)$ is satisfied for every $p \in \TT$.
Hence,
\begin{equation} \label{e pset eq 1}
p \in \pset \text{ if and only if }  \arg_0 p \in [-\pi/2, \pi/2) \ \text{ and \ $\langle \Phi^2 (x) -1 , p \rangle_\CC \le 0$ for all $x \in [a_1,\ell]$} .
\end{equation}

 Since $n=1$, we see that $\Phi [\I]= \{ 1- \om^2 m_1 (x-a_1) \ : \ x \in [a_1,\ell -a_1] \}$.
It follows easily from $\im \om^2 <0$ and $0 \not \in \Phi [\I] $ that, for $x \in (a_1, \ell]$,  there exists  a  continuous strictly increasing  branch $\xi_1 (x)$ of the multifunction
$\arg (\Phi^2 (x)-1)$ singled out by $\lim\limits_{x \to a_1+} \xi_1 (x) = \arg_0 (-\om^2)$. When $x= a_1$, we define $\xi_1 (a_1) :=  \arg_0 (-\om^2) $.
It follows from (\ref{e pset eq 1}) that
\begin{equation} \label{e pset eq 2}
p \in \pset \text{ if and only if } \arg_0 p \in [-\pi/2, \pi/2) \text{ and } [\xi_1 (a_1), \xi_1 (\ell) ] \subset [ \arg_0 p + \pi /2 , \arg_0 p + 3\pi/2]  .
\end{equation}

\emph{Step 3: (i) $\Rightarrow $ (ii)} If $\xi_1 (\ell) = \xi_1 (a_1) + \pi$, (\ref{e pset eq 2}) yields that $\pset$ consists of one number, and so a supporting line is unique, a contradiction.

\emph{Step 4: (ii)  $\Rightarrow $ (i)} Since $\pset \neq \emptyset$,  (\ref{e pset eq 2}) implies $\xi_1 (\ell) \le \xi_1 (a_1) + \pi$. Combining this with (ii), we see that
$\xi_1 (\ell) < \xi_1 (a_1) + \pi$. The latter, (\ref{e pset eq 2}), and $\xi_1 (a_1) = \im (-\om^2) > 0$ easily imply that there exist infinitely many $p \in \TT$ satisfying the right-hand
 side of the equivalence  (\ref{e pset eq 2}).
\end{proof}

\subsection{Proof of Theorem \ref{t small fr} and the case of single point mass}

\label{ss proof small fr}

\begin{lem} \label{l one support lines}
Let (\ref{e weak loc-bd}) be fulfilled and $\re \om >0$.
Assume that  $n=1$ and that there exists only one supporting line $L $ to the set $ S_0 $ at $ z_0 $.
Then $L= \RR$,  $m_1 < m$,  and $a_1 = 0$.
\end{lem}

\begin{proof}
To prove $L = \RR$, assume converse.
It follows from Proposition \ref{p z0}, Lemma \ref{l two support lines} and the lemma's proof that $m_1 = m$ and
$ \Phi^2 (\ell) -1 =  s \om^2 $ for certain $s > 0$.
Since $\Phi^2 (\ell) -1 = \om^4 m^2 (\ell - a_1)^2 -2 \om^2 m (\ell -a_1) $,
we have $ s = \om^2 m^2 (\ell - a_1)^2 -2 m (\ell -a_1) $. By (\ref{e aone<l}) ,
$a_1 < \ell$ and so $\im \om^2 = 0 $ .  This contradicts $\om \in \QQ_{IV}$.

 Let us prove $m_1 < m$ and $a_1 = 0$.
If $m_1 = m$, then Lemma \ref{l two support lines} yields $\Phi^2 (\ell) -1 \in \om^2 \RR_+$ and, in turn,
$\im \Phi^2 (\ell)  < 0 $. The latter contradicts Proposition \ref{p oinL}.
So $m_1 < m$. Theorem \ref{e dM=sum} (iii) yields  $a_1 = 0$.
\end{proof}

The following proposition essentially describes the strings consisting of a single point mass and  producing weakly* local boundary points of $\Si [\Am]$ with nonzero frequency.

\begin{prop} \label{p n=1}
Suppose (\ref{e weak loc-bd}),  $\re \om > 0$, and $n=1$.
Then $m_1 < 4 \ell$ and (exactly) one of the following two assertions holds:
\item[(i)] $m_1 = m < 4 (\ell -a_1) $ and
$\om = - \ii \frac {1}{2 (\ell - a_1) } + \sqrt{ \frac{1}{m (\ell - a_1)} -
\frac {1}{4 (\ell - a_1)^2}}$ .
\item[(ii)] $m_1<m$,  $a_1 = 0$,   and
$\om = - \frac {\ii}{2 \ell } + \sqrt{ \frac{1}{m_1 \ell } -
\frac {1}{4 \ell^2} } $ .
\end{prop}

\begin{proof}
Formula (\ref{e x0 ne l}) and $\re \om >0$ easily imply $m_1 < 4 \ell$ and, in combination with Lemmas \ref{l two support lines}-\ref{l one support lines},
the rest of the statement.
\end{proof}

\begin{prop}\label{p om for n>=2}
Suppose  (\ref{e om loc-bd}), $\re \om > 0$, and $n \ge 2$. Then
$
\re \om^2 \ge \frac{1}{m_1 (a_2 - a_1) } \ge \frac{1}{m \ell} > 0 .
$
\end{prop}

\begin{proof} By Lemma \ref{l phi on int M=0} and (\ref{e pa phi 1}),
$\pa_x \re \Phi (x) = - m_1 \re \om^2 $ for $x \in (a_1,a_2)$. It follows from Propositions \ref{p billiard} (iv) and \ref{p oinL}
that $\re \Phi (a_2) \le 0$. Since $\Phi (a_1) =1$, we see that $\re \om^2 > 0$ and
$1 \le m_1 \re \om^2 (a_2 - a_1)$.
\end{proof}

Propositions \ref{p n=1} and \ref{p om for n>=2} allows one to describe
all $\om $ and $\dd M_0$ satisfying (\ref{e weak loc-bd}), $\re \om \neq 0$, and
$\re \om^2 < \frac{1}{m \ell}$.
Straightforward calculations and Propositions \ref{p al=0 min dec}, \ref{p calc of min from w*bd} complete the proof of Theorem \ref{t small fr}.

\section{Concluding remarks}
\label{s c rem}

\textbf{1. Reduction to a problem with four real parameters.}
Let $\re \om > 0$. Assume that $\om$ is of minimal decay for $\re \om$
(or, more generally, satisfies (\ref{e weak loc-bd})  with a certain $\dd M_0$), but $\dd M_0$ and $\Phi (x) = \vphi (x,\om;\dd M_0)$ are unknown.
Then the quasi-eigenvalue $\om$, the first mass $m_1 \in \RR_+$, and the normal $p \in \{ e^{\ii \xi } \ : \ \xi \in [-\pi/2, \pi/2)\}$ from Theorem \ref{t AM opt det} completely determine
the mode $\Phi$ and the string $\dd M_0$, which can be calculated via the following procedure.

\emph{Case 1.} If $\om $ belongs to the circle $\TT_{1/m} (- \ii/m) = \{z \in \CC \ : \  |z + \ii/m | = 1/m\}$, then
$m_1=m$ and the case (i) of Proposition \ref{p n=1} takes place
(this follows from \cite[Lemma 4.1]{Ka13_KN}). So $n=1$ and $a_1$ can be recovered from the formula for $\om$, i.e.,
$a_1  = \ell + \frac {1}{2 \im \om }$. This gives $\dd M_0 = m \de (x - a_1) \dd x$ (and, in turn, gives
$\Phi$ if one needs it).

\emph{Case 2.} If $\om \not \in \TT_{1/m} (- \ii/m)$, then
Theorem \ref{t AM opt det} (ii.a) and Proposition \ref{p n=1} imply that $a_1 = 0$.
Since $m_1$, $\om$, and $p$ are known, one can find the first line segment $[\Phi (a_1), \Phi (a_2)]$ of Theorem \ref{t AM opt det} (i.a) using (\ref{e pa+ = pa- -om aj}) and Lemma \ref{l phi on int M=0}.

Indeed, consider the ray $\Ray_1 := 1 - \om^2 \RR_+$. If $\Ray_1$ does not intersect the hyperbola
$\Hyp=  \{ \pm  \sqrt{1+\ii p s } \ : \
s \in \RR \}$
(which is associated with $p$ and, in the case $p=-\ii$, is degenerate), then $n=1$ and $\dd M_0 = m_1 \de (x) \dd x$.

If $\Ray_1$ intersects $\Hyp$ at a point $1 - \om^2 m_1 \ga$ with $\ga > 0$, there are three cases. In the case $\ga > \ell$, one can see that $n=1$ and again
$\dd M_0 = m_1\de (x) \dd x$. When $\ga = \ell$, we have either $n = 2$, $a_2 = \ell  $, and $m_2>0$ given by (\ref{e mn a=l}), or $n=1$ and  $\dd M_0 = m_1\de (x) \dd x$ (if  (\ref{e mn a=l}) gives $0$ for $m_2$).
In the third case  $\ga < \ell$, one has
$n \ge 2$, $a_2 = \ga $, and the second mass $m_2$ can be found via (\ref{e mj a neq l}).
Then (in the third case) the procedure of finding of line segments of the trajectory of $\Phi$, positions $ a_j $,  and masses $ m_j $ can be continued inductively with the use of assertions (i) and (ii) of Theorem \ref{t AM opt det}.

Thus, the four parameters, $\re \om$, $\im \om$, $m_1$, and $\xi = \arg_0 p$,
completely determine $\Phi$ and $\dd M_0$.

\noindent \textbf{2. Optimizers over $\Am$ that are not extreme points of $\Si [\Am]$}.
In the case when $\alpha \neq 0$ and $\frac{1}{m \ell } - \frac {1}{4 \ell^2} < \alpha^2 \le  \frac{1}{m \ell } $, formula (\ref{e m1a1 case3}) and Proposition \ref{p n=1} imply that the unique string $\dd M^{[\alpha]}$ of minimal decay for the frequency $\alpha$ has the form  $m_1 \de (x ) \dd x$ with $0 < m_1 < m$.
So $\dd M^{[\alpha]}$ is not an extreme point of $\Am$ (but it belongs to the 2-D face
$[0\dd x \, , \, \de (x ) \dd x]$ of $\Am$).

\noindent \textbf{3. Weakly* local boundary points of $\Si [\Am]$ that are not of minimal decay over $\Am$.}
Let $m < 2\ell$ and $0 < a_1 \le \ell - \frac{m\ell}{4\ell-m}$. Then
$\om = - \ii \frac {1}{2 (\ell - a_1) } + \sqrt{ \frac{1}{m (\ell - a_1)} -
\frac {1}{4 (\ell - a_1)^2}}$ is a boundary point of $\Si [\Am]$, but is not a quasi-eigenvalue of minimal decay for corresponding frequency $\alpha =  \sqrt{ \frac{1}{m (\ell - a_1)} - \frac {1}{4 (\ell - a_1)^2}}$.

Indeed, $\om \in \Bd \Si [\Am]$ follows from  \cite[Lemma 4.1]{Ka13_KN}.
On the other side, $\om$ is not a quasi-eigenvalue of minimal decay for $\alpha $ due to  the fact that the case (ii) of Proposition \ref{p n=1} gives the smaller value $\frac{1}{2\ell}$ for the decay rate.

Obviously,  $\om$ is weakly* $\dd M_0$-local  boundary point of $\Si [\Am]$ for the associated string $\dd M_0 = m \de (x-a_1) \dd x$.

\textbf{4. Optimizers reach one of the constraints} as it is shown by the following result.

\begin{cor} \label{t constr}
Let $\dd M_0$ be of minimal decay for a certain frequency. Then at least one of the two following statements hold true: (i)  $\| \dd M_0 \|=m$, (ii) $\aone (\dd M_0) = 0$.
(In other words, $\dd M_0$ has maximal possible mass or $\dd M_0$ has maximal possible length.)
\end{cor}

This fact immediately follows from the combination of
Theorem \ref{t AM opt det} with Propositions \ref{p n=1} and \ref{p al=0 min dec}.

It is interesting to compare Corollary \ref{t constr} with numerical \cite{KS08}, \cite[the problem $\mathrm{Opt_{\max}}$]{HBKW08} and analytical \cite{Ka13,Ka13_KN}  results on the quasi-eigenvalue optimization with quite different constraints on coefficients. These results also lead to the conclusion that at least one of the constraints is reached for optimizers.

\section{Appendix: the proof of Lemma \ref{l 2par per}}
\label{s Appendix}

We use several lemmas from \cite{Ka13}. Lemma \ref{l P ser} is \cite[Lemma 3.4]{Ka13},
Lemmas \ref{l as z} and \ref{l simple z in T}  can be easily extracted from the proof of \cite[Lemma 3.6]{Ka13}.

\begin{lem}[\cite{Ka13}] \label{l P ser}
Let $P(z,\tau) = z^r + h_1 (\tau) z^{r-1} + \dots + h_r (\tau)$ be a monic polynomial in $z\in \CC$ with coefficients
$h_j$ analytic in the complex variable $\tau$ for $\tau \in \DD_\delta (0)$ (with  $\delta>0$). Suppose
\[
h_j (0) = 0 \ \ \text{ for } j = 1, \dots, r, \ \ \ \text{ and } \ h_r '(0) \neq 0.
\]
Then for $\tau$ close enough to $0$ there exist
exactly $r$ distinct roots of $P (\cdot,\tau)= 0$ and these roots are given by an
$r$-valued analytic  function $Z (\tau)$ that admits a Puiseux series representation
\begin{equation} \label{e Z Pser}
Z (\tau) = \sum_{j=1}^{\infty} c_j \tau^{j/r} \  \
\text{ with the leading coefficient } \ c_1 = \sqrt[\rr]{-h'_r (0)} \neq 0
\end{equation}
(here $\sqrt[\rr]{\tau}$ is an arbitrary  branch of the multi-function $\tau^{1/r}$, and $c_j \in \CC$ are constants).
\end{lem}

It is not an essential restriction to assume that $\sqrt[\rr]{\cdot}$ and $\arg z$
are continuous in $\CC \setminus \overline{\RR_-} $ and fixed by $\sqrt[\rr]{1}=1$, $\arg 1 =0$, and that
the complex numbers $\eta_j$ defined by (\ref{e def eta}) satisfy
\begin{eqnarray}
\arg\sqrt[r]{\eta_j} = (-1)^j \frac{\xi_0}{2r}, \ \ \ j=1,2 . \label{e arg pa Q}
\end{eqnarray}

By the Weierstrass preparation theorem, $Q (z;\zeta) = P (z;\zeta) R (z;\zeta)$
in a certain nonempty polydisc $\DD_{\ep_1} (0) \times \DD_{\de_1} (0) \times \DD_{\de_1} (0)$ with a Weierstrass polynomial
\begin{equation} \label{e Qzga}
 P (z;\zeta) = z^r + q_1 (\zeta) z^{r-1} + \dots + q_r (\zeta) ,
\end{equation}
and functions $R$ and $q_j$ such that:
\begin{itemize}
\item the function $R$ is analytic and has no zeroes in $\DD_{\ep_1} (0) \times \DD_{\de_1} (0) \times
\DD_{\de_1} (0)$,
\item the coefficients $q_j $ of \emph{the Weierstrass polynomial} $P$ are analytic
in $\DD_{\de_1} (0) \times \DD_{\de_1} (0)$,
\item $q_j (0,0) = 0$ for all $j=1,\dots,r$.
\end{itemize}
Then
\begin{equation} \label{e zer Q P}
Q \text{ and } P \ \text{ have the same zeroes in } \DD_{\ep_1} (0) \times \DD_{\de_1} (0) \times
\DD_{\de_1} (0) .
\end{equation}

Let $2\de_2 < \xi_0 /r$. Then there exist $\theta_1$, $\theta_2$ such that  $0<\theta_1<\theta_2<1$
and
\begin{eqnarray} \label{e <xi1<}
\arg \sqrt[\rr]{\wt \eta_1}  < \arg \sqrt[\rr]{\eta_1} + \de_2 < \arg \sqrt[\rr]{\eta_2} - \de_2  < \arg \sqrt[\rr]{\wt \eta_2},
\text{ where } \wt \eta_j := (1-\theta_j) \eta_1 + \theta_j \eta_2 .
\end{eqnarray}
Let us define the (real) triangles
\[
T_\de [\theta_1, \theta_2] := \left\{ \ ( [1-\theta] c , \theta c )  \subset \CC^2 \, : \, c \in (0,\delta] , \ \theta \in [\theta_1,\theta_2]
\right\} .
\]
So $T_{\de_1} [\theta_1, \theta_2] \subsetneq T_{\de_1}$. The following property of $T_\de [\theta_1, \theta_2]$ is essential for the next lemma:
if a sequence $\{ \zeta^{[n]} \}_{n=1}^{\infty} \subset T_\de [\theta_1, \theta_2] $ tends to $ \{ 0,0 \} $, then
\begin{equation} \label{e ga asymp}
\zeta_1^{[n]} \asymp \zeta_2^{[n]} \asymp |\zeta^{[n]}| \ \text{ as } n \to \infty .
\end{equation}
Here and below by $\zeta^{[n]}$ we denote the pair $( \zeta_1^{[n]}, \zeta_2^{[n]} ) \in \CC^2$.

\begin{lem}[\cite{Ka13}] \label{l as z}
Assume that there exist sequences $z^{[n]}$, $\zeta_1^{[n]}$ and $\zeta_2^{[n]} $ such that
\item[(i)] $\zeta^{[n]} = ( \zeta_1^{[n]}, \zeta_2^{[n]} ) $ belong to $T_{\de_1} [\theta_1, \theta_2]$
and tend to $( 0,0 )$ as $n \to \infty$,
\item[(ii)] $P (z^{[n]}; \zeta^{[n]} ) = 0$ for all $n \in \NN$.

Then $z^{[n]} \to 0$. Moreover,
\begin{equation} \label{e as zn}
(z^{[n]})^r = \left( \zeta_1^{[n]} \eta_1 + \zeta_2^{[n]} \eta_2 \right) \ [1+o(1)] \ \text{ and } \ |z^{[n]}| \asymp |\zeta^{[n]}|^{1/r}
\ \text{ as } n\to \infty .
\end{equation}
\end{lem}

We include the proof of Lemma \ref{l as z} since it explains the definition of $\eta_{1,2}$,
and assumption  (\ref{a eta}).
To prove the lemma, it is enough to note that  (ii) and (\ref{e Qzga}) yields that $z^{[n]} = o(1)$ as $n \to \infty$.
Plugging this back to (ii), one gets
\[
(z^{[n]})^r = - q_r (\zeta^{[n]}) + q_{r-1} (\zeta^{[n]}) o (1) + \dots + q_{1} (\zeta^{[n]}) o (1)  .
\]
So (\ref{e ga asymp}) and $q_j (0,0) = 0$ imply
\[
(z^{[n]})^r = - \zeta_1^{[n]} \pa_{\zeta_1} q_r (0,0) - \zeta_2^{[n]} \pa_{\zeta_2} q_r (0,0)
+ o(|\zeta^{[n]}|) .
\]
The differentiation of the equality $Q=PR$ and formula (\ref{e Qzga}) easily gives
$ \pa_{\zeta_j} q_r (0,0) = - \eta_j $, $ j=1,2$.
Thus, (\ref{e as zn}) follows from (\ref{e ga asymp}) and assumption (\ref{a eta}).

\begin{lem}[\cite{Ka13}]  \label{l simple z in T}
Let  $0<\theta_1<\theta_2<1$.
There exists $\epsilon \in (0,\de_1)$ with the following property:
if $P (z,\zeta) = 0$ and $\zeta \in T_{\epsilon} [\theta_1,\theta_2]$, then
$z$ is a simple zero of the polynomial $P (\cdot,\zeta) $.
\end{lem}

Let $\epsilon$ be as in Lemma \ref{l simple z in T}.
It is essential for the next lemma that $(0,0) \not \in T_\epsilon [\theta_1, \theta_2] $.
So there exist open neighborhoods (in the sense of $\CC^2$) of $T_\epsilon [\theta_1, \theta_2]$ that do not contain the origin $(0,0)$.

\begin{lem} \label{l Z1}
There exist an open simply connected set $W \subset \CC^2$ and an analytic in $W$ function
$Z_1 (\zeta_1,\zeta_2)$ with the following properties:
\item[(i)] $T_\epsilon [\theta_1, \theta_2] \subset W$,
\item[(ii)] for each $\zeta \in T_\epsilon [\theta_1, \theta_2]$, the number
$Z_1 (\zeta) = Z_1 (\zeta_1,\zeta_2)$ is a zero of $P ( \cdot , \zeta_1, \zeta_2 )$ and
\begin{equation} \label{e Z1 as}
Z_1 (\zeta) = \sqrt[r]{\zeta_1 \eta_1 + \zeta_2 \eta_2 + o (\zeta)} \ \  \text{ as } \zeta \to 0 ,
\end{equation}
\item[(iii)] for every $\theta \in [\theta_1, \theta_2]$ there exists $\epsilon_1 (\theta) >0$
 such that the series representation
\begin{equation} \label{e Ztau}
Z_1 \left( [1-\theta] \tau , \theta \tau \right) = \sum_{j=1}^{\infty} c_j (\theta) \ (\sqrt[r]{\tau})^j  \ \
\text{ with the leading coefficient }
c_1 (\theta) = \sqrt[\rr]{ (1-\theta) \eta_1 + \theta \eta_2 }
\end{equation}
holds for $\tau \in [0, \epsilon_1 (\theta)] $.
\end{lem}

\begin{proof}
Applying Lemma \ref{l P ser} to the function $\wt P (z,\tau) := P (z, (1-\theta) \tau , \theta \tau)$
with  complex variables $z$ and $\tau$ and a fixed parameter $\theta \in [\theta_1,\theta_2]$,
one can produce the $r$-valued Puiseux series for the zeroes
$Z ([1-\theta] \tau, \theta \tau )$ of $P (\cdot, (1-\theta) \tau , \theta \tau)$:
\begin{equation} \label{e Ztau 2}
Z \left( [1-\theta] \tau , \theta \tau \right) = \sum_{j=1}^{\infty} c_j (\theta) \tau^{j/r} , \  \
\text{ with }
c_1 (\theta)= \sqrt[r]{ (1-\theta) \eta_1 + \theta \eta_2 } .
\end{equation}
Choosing  $\theta_0 \in [\theta_1,\theta_2]$
and placing the branch $ \sqrt[r]{\tau} $ (chosen at the beginning of this subsection) instead of the multi-function $\tau^{1/r}$
in (\ref{e Ztau 2}),
we obtain a branch $Z_1 \left( [1-\theta_0] \tau , \theta_0 \tau \right)$
of the multi-function $ Z \left( [1-\theta_0] \tau , \theta_0 \tau \right) $.
For $\tau \in [0, \epsilon_1 (\theta_0)]$ with $\epsilon_1 (\theta_0)>0$ small enough,
the series converges and indeed gives a zero of $P(\cdot, (1-\theta_0) \tau , \theta_0 \tau)$.

Since $T_\epsilon [\theta_1, \theta_2] $ is simply connected,
Lemma \ref{l simple z in T} and the implicit function theorem for simple zeroes
imply that $Z_1$ can be extended from the line-segment
\[
\left\{ \  ( [1-\theta_0] \tau , \theta_0 \tau) \in \CC^2  \ : \ \tau \in [0, \epsilon_1 (\theta)] \ \right\}
\]
to an analytic function $Z_1 (\zeta_1,\zeta_2)$ in a certain open simply connected
neighborhood $W$ of $T_\epsilon [\theta_1, \theta_2] $ saving the property $P \left( Z_1 (\zeta); \zeta \right) = 0$.
If such a neighborhood $W$ is fixed, the extension is unique. (Note that $T_\epsilon [\theta_1, \theta_2] $ does not contain
the origin and, in the case $r>1$, the neighborhood $U$ also must not contain the origin).

Lemma \ref{l as z} implies that $Z_1 (\zeta_1,\zeta_2) \to 0$ as
$\zeta \to 0$ in $T_\epsilon [\theta_1, \theta_2] $.
That is, $Z_1$ is continuous in the closure $\overline{T_\epsilon [\theta_1, \theta_2]}$.
(Note that $Z_1 (0,0) = 0$ due to (\ref{e Ztau 2}) and that $\overline{T_\epsilon [\theta_1, \theta_2]}$
consists of $(0,0)$ and $T_\epsilon [\theta_1, \theta_2]$).
Similar to the proof of Lemma \ref{l as z}, one can show that
\begin{equation} \label{e Z1= +f}
Z_1^r (\zeta) = \zeta_1 \eta_1 + \zeta_2 \eta_2 + f (\zeta) , \
\end{equation}
 where $f$ is analytic in $U$ and $ f (\zeta) = o (\zeta) $ as $\zeta \to 0 $ in $T_\epsilon [\theta_1, \theta_2]$.
So it is possible to take $\epsilon_2<\epsilon$ such that
$\zeta \in T_{\epsilon_2} [\theta_1, \theta_2] $ implies
\begin{equation} \label{e Zr sec}
 Z_1^r (\zeta) \in \Sec [\arg \eta_1 , \arg \eta_2 ] \ \ \ \text{ and, in particular, } Z_1 (\zeta) \neq 0
\end{equation}
(note that, by definition (\ref{e Sec def}), the sector $\Sec [\arg \eta_1 , \arg \eta_2 ]$ does not contain $0$).
From this and the expression for $c_1 (\theta_0)$ in (\ref{e Ztau 2}), we get
\begin{equation} \label{e sqrtZr=Z}
\sqrt[r]{ Z_1^r \left( \zeta \right) } = Z_1 \left( \zeta \right)
\end{equation}
for all $\zeta = ( [1-\theta_0] \tau , \theta_0 \tau ) $ with $  0 < \tau < \min \{ \epsilon_2, \epsilon_1 (\theta_0)\} $.

Now (\ref{e Z1 as}) follows from (\ref{e Z1= +f}) and the fact that formula (\ref{e sqrtZr=Z}) survives the above considered analytic
extension to $T_{\epsilon_2} [\theta_1, \theta_2] $. Indeed,
assume that for certain $\zeta \in T_{\epsilon_2} [\theta_1,
\theta_2]$ and a certain nontrivial $r$-th root of unity $1^{1/r}$ ($1^{1/r} \neq 1$)
the equality $Z_1 ( \zeta ) = 1^{1/r} \sqrt[r]{
Z_1^r ( \zeta) } $ holds.
Then (\ref{e Zr sec}) and the standard argument concerning simultaneously closed and open subsets
of a connected set imply that $Z_1$ is discontinuous at some point $\wt \zeta \in T_{\epsilon_2} [\theta_1, \theta_2] $,
a contradiction.

For every $\theta \in [\theta_1, \theta_2]$, there exists $\epsilon_1 (\theta) >0$ and a branch $\sqrt[r]{_{_\theta} \tau}$
of $\tau^{1/r}$ such that the function
$Z_1 \left( [1-\theta] \tau , \theta \tau \right)$
admits the series representation (\ref{e Ztau 2}) for $\tau \in [0, \epsilon_1 (\theta)] $
with $(\cdot)^{1/r}$ replaced by the branch $\sqrt[r]{_{_\theta} z}$.
(In other words, for a fixed $\theta$ formula (\ref{e Ztau 2}) holds with a fixed branch $\sqrt[r]{_{_\theta} \tau}$
of $\tau^{1/r}$; note that for $\theta=\theta_0$ this fact holds by definition of $Z_1$).
Indeed, for $\tau\leq \epsilon_1 (\theta) $ with $\epsilon_1 (\theta)$ small enough,
Lemma \ref{l P ser} implies that all the roots of $\wt P (z,\tau) = 0$ are distinct
and are produced by the $r$-valued series (\ref{e Ztau 2}).
Assume for a moment that for two different numbers $\tau_1, \tau_2 \in (0, \epsilon_1 (\theta)] $, the value of
$Z_1 \left( [1-\theta] \tau , \theta \tau \right)$ is given by
(\ref{e Ztau 2}) with different branches of $(\cdot)^{1/r}$. Then standard arguments imply that
$Z_1 \left( [1-\theta] \tau , \theta \tau \right)$ is not continuous on $(0, \epsilon_1 (\theta)]$.
This contradicts the definition of $Z_1$.

Finally, comparing asymptotics of (\ref{e Z1 as}) with that of
the first term in the series (\ref{e Ztau 2}), we see
that $Z_1 \left( [1-\theta] \tau , \theta \tau \right)$ corresponds to the branch $ \sqrt[r]{\tau} $
(i.e., $\sqrt[r]{_{_\theta} \tau} = \sqrt[r]{\tau}$ for all $\theta \in [\theta_1, \theta_2]$).
\end{proof}

\begin{lem} \label{l img Z1}
For $\ep>0$ small enough, the image $Z_1 [ \, T_{\de_1} [\theta_1, \theta_2] \, ]$ of $\ T_{\de_1} [\theta_1, \theta_2]$
contains the set
\[
S_1 := \DD_\ep (0) \cap \Sec \left[ \arg \sqrt[r]{\eta_1} + \delta_2 ,  \arg
\sqrt[r]{\eta_2} - \delta_2 \right] .
\]
\end{lem}

\begin{proof}
For $\tau > 0$, consider loops (closed paths) $\ga_\tau : [0,1] \to \RR^2 \subset \CC^2$ that are parameterizations of
the boundaries of the triangles $T_\tau [\theta_1,\theta_2]$ and have the following property:
the map $\ga (s,\tau) = \ga_\tau (s)$ defined on $ [0,1] \times [0,\nu] $ is a homotopy of $\ga_\nu$ into the loop
$\ga_0  \equiv 0$ ($\nu$ is arbitrary positive, $\ga_0$ is defined by $\ga_0 (s) = 0$ for all $s \in [0,1]$, and so
the loop $\ga_0$ degenerates to the origin $( 0, 0)$). To be specific,
put
\[
\gamma_\tau (s) := \left\{
\begin{array}{lll}
t_1 (s) \ (1-\theta_1 , \theta_1 ) & \text{ for } s \in [0,1/3]  & \text{ with } t_1 (s) := 3 \tau  s ,  \\
\tau \ ( 1-t_2 (s) , t_2(s) ) & \text{ for } s \in [1/3,2/3] & \text{ with }
t_2 (s) := (\theta_2-\theta_1) (3s-1) + \theta_1 , \\
t_3 (s) \ ( 1-\theta_2 , \theta_2 ) & \text{ for }  s \in [2/3,1] & \text{ with } t_3 (s) := 3 \tau (1-s)
\end{array}
\right.
\]

$Z_1 (\ga_\tau)$ are loops in $\CC$ (these loops are rectifiable in the sense that
$\pa_s Z_1 (\ga_\tau (\cdot)) \in L_{\CC}^1 [0,1]$,
the loop $Z_1 (\ga_0) \equiv 0$ is reduced to the point $0$). For any fixed $\nu \geq 0$, the map
$Z_1 (\ga (s,\tau)) : [0,1] \times [0,\nu] \to \CC$ is a loop homotopy of $Z_1 (\ga_\nu)$ into $0$.

By inequalities (\ref{e <xi1<}) and Lemma \ref{l Z1} (see also the proof of Lemma \ref{l Z1}),
there exists $\nu \in (0, \delta_1]$ such that
\begin{itemize}
\item
for $\zeta \in T_{\nu} [\theta_1, \theta_2]$,
\begin{equation} \label{e Z1 sec}
 Z_1 (\zeta) \neq 0  \text{ and, moreover, } Z_1 (\zeta) \in \Sec [\arg \sqrt[r]{\eta_1} , \arg \sqrt[r]{\eta_2} ] ;
\end{equation}
\item
for $\tau \in (0,\nu]$,
\begin{eqnarray} \label{e argZ1 theta12}
\arg Z_1 \left( [1-\theta_1]\tau, \theta_1 \tau \right)
 < \arg \sqrt[\rr]{\eta_1} + \de_2 < \arg \sqrt[\rr]{\eta_2} - \de_2
 < \arg Z_1 \left( [1-\theta_2]\tau, \theta_2 \tau \right) .
\end{eqnarray}
\end{itemize}

Take $\ep$ such that
\begin{equation} \label{e ep<min}
0<\ep < \min_{\theta \in [\theta_1, \theta_2]} | Z_1 \left( [1-\theta]\nu , \theta \nu \right) |.
\end{equation}
Then for arbitrary $z_0 \in S_1$ the index
of the point $z_0$ w.r.t. the loop $Z_1 (\ga_\nu)$ equals 1,
\begin{equation} \label{e ind}
\ind [z_0 ; Z_1 (\ga_\nu) ] = 1 .
\end{equation}
Indeed, by (\ref{e Z1 sec}), (\ref{e argZ1 theta12}), and (\ref{e ep<min}),
the increments of the argument of $Z_1 (\ga_\nu (s))$ on each of the intervals $s \in [0,1/3]$,
$s \in [1/3,2/3]$, and $s \in [2/3,1]$ are positive, but less than $\pi$, $2\pi$, and $\pi$, respectively.

It follows from (\ref{e ind}) that $z_0 $ belongs to the image $S_2 := Z_1 (T_\nu [\theta_1, \theta_2])$. In fact,
assume that $z_0 \not \in S_2$. Then $z_0 $ does not belong to the curve (image of the loop)
$Z_1 (\ga_\tau )$ for every $\tau \in [0,\nu]$.  So $Z_1 (\ga (s,\tau))$ is a homotopy of $Z_1 (\ga_\nu)$ into
$Z_1 (\ga_0) \equiv 0$
in the domain $\CC \setminus \{ z_0 \}$. Thus, $\ind [z_0 ; Z_1 (\ga_\nu) ] = \ind (z_0 ; Z_1 [\ga_0]) = 0$.
This contradicts (\ref{e ind}).
\end{proof}

Summarizing, we see from Lemma \ref{l Z1} that  $P ( Z_1 (\zeta) ; \zeta ) = 0$ and from (\ref{e zer Q P}) that
$Q ( Z_1 (\zeta) ; \zeta ) = 0$ for all $\zeta \in T_{\de_1} [\theta_1, \theta_2]$. Thus,
$
Z_1 [ \, T_{\de_1} [\theta_1, \theta_2] \, ] \subset \Si_Q \left[ \, T_{\de_1} [\theta_1, \theta_2] \, \right] \subset
\Si_Q \left[ T_{\de_1} \right]$ .
Lemma \ref{l img Z1} completes the proof.

\end{document}